\documentclass[final, leqno]{amsart}

\usepackage{amsmath,amsthm}
\usepackage {latexsym}
\usepackage{amssymb}

\newcommand{\eps}{{\varepsilon}}
\newcommand{\R}{{\mathbb R}}

\newcommand{\Compl}{{\mathbb C}}
\newcommand{\meas}{{\mathcal M}}
\newcommand{\B}{{\mathcal B}}
\newcommand{\U}{{\mathcal U}}

\newcommand{\X}{\mathcal X}

\newcommand{\les}{\lesssim}

\newcommand{\Kato}{{\mathcal K}}

\newcommand{\la}{\langle}
\newcommand{\ra}{\rangle}
\newcommand{\1}{{\mathbf 1}}
\def\norm[#1][#2]{\|#1\|_{#2}}
\def\bignorm[#1][#2]{\big\|#1\big\|_{#2}}
\def\Bignorm[#1][#2]{\Big\|#1\Big\|_{#2}}
\def\japanese[#1]{\langle #1 \rangle}
\def\Im[#1]{{\rm Im}(#1)}
\def\Re[#1]{{\rm Re}(#1)}

\newcommand{\be}{\begin{equation}}
\newcommand{\ee}{\end{equation}}
\newcommand{\lb}{\label}
\newcommand{\dd}{\,d}
\newcommand{\K}{\Kato}

\newcommand{\mb}{\mathbf}

\newcommand{\M}{\mathbb M}
\newcommand{\dl}{\nabla}
\DeclareMathOperator*{\esssup}{ess\,sup}
\DeclareMathOperator{\sgn}{sgn}

\newtheorem{theorem}{Theorem}
\newtheorem{lemma}[theorem]{Lemma}

\newtheorem{corollary}[theorem]{Corollary}

\newtheorem{proposition}[theorem]{Proposition}

\theoremstyle{remark}
\newtheorem{remark}{Remark}
\newtheorem{definition}{Definition}

\begin{document}

\title[Strichartz Estimates for the Wave Equation]{Strichartz Estimates and 
Maximal Operators for the Wave Equation in $\R^3$}

\author{Marius\ Beceanu}
\thanks{The first author received support from the Rutgers Research Council 
during the preparation of this work.}
\address{School of Mathematics, Institute for Advanced Study, Princeton, NJ, 08540}
\email{mbeceanu@math.rutgers.edu}

\author{Michael\ Goldberg}
\thanks{The second author received support from NSF grant DMS-1002515 during the 
preparation of this work.}
\address{Department of Mathematics, University of Cincinnati,
Cincinnati, OH 45221-0025}
\email{Michael.Goldberg@uc.edu}

\begin{abstract}
We prove sharp Strichartz-type estimates in three dimensions, including some which hold in 
reverse spacetime norms, for the wave equation with potential. These results are 
also tied to maximal operator estimates studied by Rogers--Villaroya, of which 
we prove a sharper version.

As a sample application, we use these results to prove the local well-posedness
and the global well-posedness for small initial data of semilinear 
wave equations in $\R^3$ with quintic or higher monomial nonlinearities.
\end{abstract}

\maketitle

\section{Introduction} \label{sec:intro}

\subsection{Main result}
Consider the linear wave equation in $\R^3$
\begin{equation} \label{eq:wave}
\partial^2_t f - \Delta f + V f = F, f(0) = f_0,\ \partial_t f(0) = f_1.
\end{equation}

Under rather general conditions (i.e.\ if $H={-\Delta}+V$ is selfadjoint on
$L^2$), taking initial data in $L^2 \times \dot H^{-1}$ for example, the 
solution to this equation is given by the formula
\begin{equation} \label{eq:formalsolution}
f(t) = \cos(t\sqrt H) f_0 + \frac {\sin (t\sqrt H)}{\sqrt H} f_1 + \int_0^t 
\frac {\sin ((t-s)\sqrt H)}{\sqrt H} F(s) \dd s.
\end{equation}
The expressions in~\eqref{eq:formalsolution} are independent of the branch 
chosen for $\sqrt{H}$, 
so are well defined even if $H$ is not positive.

The homogeneous equation has a conserved energy, namely $E(t) = 
\norm[f][\dot{H}^1]^2
+ \norm[\partial_t f][2]^2 + \int_{\R^3} V|f|^2\,dx$,
which however does not preclude time-decay of solutions with respect to other 
norms.
In the free ($V = 0$) case there are well known uniform decay bounds when the initial 
data $f_0$ and
$f_1$ possess a sufficient degree of Sobolev regularity: 

\begin{align*}
\norm[\cos(t\sqrt{-\Delta})f_0][\infty] &\les |t|^{-1} \norm[f_0][W^{2,1}(\R^3)] 
\\
\Bignorm[\frac{\sin(t\sqrt{-\Delta})}{\sqrt{-\Delta}}f_1][\infty] &\les |t|^{-1}
\norm[f_1][W^{1,1}(\R^3)]
\end{align*}

These operators act by pointwise multiplication of the Fourier transform, so 
their mapping
properties on $L^2$ and $H^s$ follow immediately by Placherel's identity.  By 
combining them
with the dispersive estimates above, one can extract the family of Strichartz 
inequalities that
control each term in~\eqref{eq:formalsolution} with respect to certain norms 
$L^p_tL^q_x$.
The full range of valid pairs of exponents is determined in~\cite{tao}.

We will prove dispersive and Strichartz estimates for a large class of
short-range potentials
$V(x)$, and also establish bounds of a similar nature in the reversed space-time 
norms
\begin{equation*}
\norm[f][L^p_x L^q_t] = \bigg(\int_{R^d} \|f(x, t)\|_{L^q_t}^p \dd 
x\bigg)^{1/p}.
\end{equation*}

Two applications of the reversed Strichartz inequalities are presented.  For the 
linear evolution
we prove an endpoint estimate for the maximal operator, and deduce almost-everywhere
convergence to the initial data when $(f_0, f_1)$ belong to the energy space.  
These results are
new even when $V = 0$.  We then state a concise global well-posedness result for 
small solutions
of the energy-critical semilinear wave equation with a potential.

One substantial difference between the perturbed Hamiltonian $H = -\Delta + V$ 
and the Laplacian
is the possible existence of eigenvalues.  For the class of short-range 
potentials we consider, the
essential spectrum of $H$ is $[0,\infty)$ and the point spectrum may include a 
countable number of
eigenvalues occupying a bounded subset of the real axis that is discrete away 
from zero.  Embedded
positive eigenvalues do not occur if $V \in L^{3/2}_{\rm loc}$~\cite{IoJe}; when 
$V$ is more singular
we add this as an assumption.  We further assume that zero is a regular point of 
the spectrum of $H$. 
Under these hypotheses $H$ possesses pure absolutely continuous spectrum on 
$[0,\infty)$ and a 
finite number of negative eigenvalues.

If $E < 0 $ is a negative eigenvalue, the associated eigenfunction responds to 
the wave
equation propagators via scalar multiplication by $\cos(t\sqrt{E})$ or $E^{-
1/2}\sin(t\sqrt{E})$, 
both of which grow exponentially due to $\sqrt{E}$ being purely imaginary.  
Dispersive estimates for $H$
must include a projection onto the continuous spectrum in order to avoid this 
growth.  Otherwise, further conditions on the potential are required to prevent the 
existence of eigenvalues
entirely.

Results below refer to the Kato norm $\Kato$, introduced by Rodnianski--Schlag 
in \cite{rodsch} and D'Ancona--Pierfelice in \cite{DaPi}. 
In particular, potentials $V$ are taken in the Kato norm 
closure of the set of bounded, compactly supported functions, which we denote by 
$\Kato_0$.

\begin{definition}
The Kato space is the Banach space of measures with the property that
\begin{equation}
\norm[V][\Kato] := \sup_{y\in\R^3} \int_{\R^3} \frac{|V(x)|}{|x-y|}\,dx < 
\infty.
\end{equation}
\end{definition}
$\K_0$ contains the Lorentz space $L^{3/2, 1}$ by Young's inequality, as well 
as $\dot W^{1, 1}$ by Lemma~\ref{kato_w11}.  Among compactly supported
functions, $K_0$ coincides with the Kato class defined by the property
$\sup_y \int_{|x-y| < r} |V(x)|/|x-y|\,dx \to 0$ as $r$ decreases to zero (see
\cite[Lemma 4.4]{DaPi}).

Note that the homogeneous wave equation~\eqref{eq:wave} 
is left unchanged by the rescaling $f(x, t) \mapsto f(\alpha x, \alpha t)$, as 
long as the potential $V(x)$ changes to $\alpha^2 V(\alpha x)$. Quadratic decay 
at infinity is invariant under this rescaling, as is the Kato norm.  At the
level of operators, pointwise multiplication by $V \in \Kato_0$ is 
infinitesimally form-bounded relative to the Laplacian,
hence there is a unique self-adjoint realization of $H = -\Delta + V$.

Several estimates in the following discussion use the complex interpolation 
spaces $\Kato^\theta:= (L^1, \Kato)_{[\theta]}$, for $0 \leq \theta \leq 1$, and 
their duals. Note that
$$
\Kato^\theta := (L^1, \Kato)_{[\theta]} = \{f \mid \sup_x \int_{\R^3} \frac 
{|f(y)| \dd y}{|x-y|^{\theta}}<\infty\}
$$
and by Young's inequality
\be\lb{young}
L^{3/(3-\theta), 1} \subset \Kato^\theta.
\ee
The description of $(\Kato^\theta)^*$ is more involved:
$$
(\Kato^\theta)^* := \{f \mid f(x) = \int g(x, y) \dd \mu(y),\ \int_{\R^3} 
\big(\sup_x |x-y|^{\theta} |g(x, y)|\big) \dd \mu(y) < \infty \}.$$
Also note that by duality, for $0 \leq \theta \leq 1$
\be\lb{young*}
(\Kato^{\theta})^* \subset L^{3/\theta, \infty}.
\ee
Thus results expressed by means of $\Kato^{\theta}$ and $(\Kato^\theta)^*$ are 
sharper than those containing the scale of Lorentz spaces $L^{p, q}$.

In this paper we exhibit a new class of Strichartz inequalities in $\R^3$, which 
hold in reversed space-time norms of the form
$$
\|f\|_{L^p_x L^q_t} = \bigg(\int_{R^d} \|f(x, t)\|_{L^q_t}^p \dd x\bigg)^{1/p}
$$
and $\|f\|_{\Kato^\theta_x L^p_t}$, $\|f\|_{(\Kato^\theta)^*_x L^p_t}$ defined 
in the same manner.
For pairs of exponents $(p,q)$ with $q >p$ such inequalities will be stronger 
than the standard Strichartz bounds. In other cases, reversed space-time
estimates will hold for pairs of coefficients for which the corresponding 
regular
Strichartz estimates are false.

We allow for inhomogeneous terms $F$ of the same types, which enables the 
familiar bootstrapping methods 
for semilinear equations to take place entirely in reversed space-time function 
spaces.

In the following discussion let $a \les b$ denote $|a| \leq C |b|$ for various 
values of $C$.

The assumption that $H$ has no eigenvalues or resonances at zero leads to a 
number of equivalences between
Sobolev spaces based on applications of $H$ in place of the Laplacian.  
Lemma~\ref{comp_lemma} demonstrates that
$\norm[\Delta f][1]$ and $\norm[Hf][1]$ are equivalent under these conditions, 
and the same applies to
$\norm[\Delta f][\Kato]$ and $\norm[H f][\Kato]$ (or 
$\norm[\,\cdot\,][L^{3/2,1}]$, if $V \in L^{3/2,1}$).

It is also true that the positive quadratic form $\la |H|f, f\ra$ is equivalent 
to $\norm[f][\dot{H}^1]^2$, though this fact
will play a less prominent role in our analysis.

The main result of this paper is the following:
\begin{theorem}[Reversed-norm dispersive estimates] \lb{main_theorem} Consider a 
real-valued potential $V \in \Kato_0$ on $\R^3$ such that ${-\Delta} + V$ has 
no eigenvalues on $[0, \infty)$ and no resonance at zero. Then
$$\begin{aligned}
\Big\|t \frac {\sin(t \sqrt H) P_c}{\sqrt H} f \Big\|_{L^{\infty}_x L^1_t} 
  &\les \norm[f][1], \\
\Big\|\int_t^{\infty} \Big|\frac {\sin(s \sqrt H) P_c}{\sqrt H} f \Big|\dd s 
\Big\|_{L^{\infty}_x L^1_t} &\les \|f\|_1, \\
\esssup_x \int_t^{\infty} \Big|\frac {\sin(s\sqrt H) P_c}{\sqrt H} f \Big| \dd s 
&\les |t|^{-1} \|f\|_1
\end{aligned}$$
and also
$$\begin{aligned}
\Big\|\frac {\sin(t \sqrt H) P_c}{\sqrt H} f \Big\|_{L^{\infty}_x L^1_t} 
&\les \norm[f][\Kato] \les \norm[\nabla f][1],\\
\Big\|\frac {\sin(t \sqrt H) P_c}{\sqrt H} f \Big\|_{\Kato^*_x L^1_t} 
& \les \norm[f][1].
\end{aligned}$$
Assume that $f \in L^2$ and $\nabla f \in L^1$.  Then
\begin{align*}
\norm[t\cos(t\sqrt{H})P_c f][L^\infty_xL^1_t] & \les \norm[\nabla f][1] \\
\norm[\cos(t\sqrt{H})P_c f][\Kato^*_xL^1_t] & \les \norm[\nabla f][1] \\
\Bignorm[\frac{\sin(t\sqrt{H})P_c}{\sqrt{H}} f][\Kato^*_x L^\infty_t]
&\les \norm[\nabla f][1] \\
\Bignorm[\frac{\sin(t\sqrt{H})P_c}{\sqrt{H}} f][L^\infty_x]
&\les |t|^{-1}\norm[\nabla f][1].
\end{align*}
Assume that $f \in L^2$ and $\Delta f \in L^1$ or $\dl f \in \Kato$. Then
$$\begin{aligned}
\|\cos(t \sqrt H) P_c f \|_{L^{\infty}_x L^1_t} &\les \min(\norm[\Delta f][1], 
\|\dl f\|_{\Kato}),\\
\|\cos(t \sqrt H) P_c f\|_{L^{\infty}_x} &\les |t|^{-1} \|\Delta f\|_1,\\ 
\Big\|\frac {\sin(t \sqrt H) P_c}{\sqrt H} f\Big\|_{L^{\infty}_t L^{\infty}_x} 
&\les \min(\|\Delta f\|_1, \|\dl f\|_{\Kato}).
\end{aligned}$$
Furthermore,
$$\begin{aligned}
\|\cos(t \sqrt H) P_c f\|_{\Kato^*_x L^{\infty}_t} &\les \|\Delta f\|_{1} \\
\|\cos(t \sqrt H) P_c f\|_{L^{\infty}_{x, t}} &\les \|\Delta f\|_{\Kato}.\\
\end{aligned}$$
The resulting inhomogeneous estimates are, for $1\leq p\leq \infty$,
$$\begin{aligned}
\Big\|\int_{t'<t} \frac {\sin((t-t') \sqrt H) P_c}{\sqrt H} F(t') 
\dd t' \Big\|_{\Kato^*_x L^p_t} 
  &\les \|F\|_{L^1_x L^p_t},\\
\Big\|\int_{t'<t} \frac {\sin((t-t') \sqrt H) P_c}{\sqrt H} F(t') 
\dd t' \Big\|_{L^{\infty}_x L^p_t} 
  &\les \|F\|_{\Kato_x L^p_t}. \\
\end{aligned}$$
Likewise,
$$\begin{aligned}
\Big\|\int_{t' < t} \frac{\cos((t-t')\sqrt {H}) P_c}{H} F(t') \dd 
t' \Big\|_{L^{\infty}_xL^p_t} &\les \|F\|_{L^1_x L^p_t},\\
\Big\|\int_{t' < t} \frac{\cos((t-t')\sqrt {H}) P_c}{H} F(t') \dd 
t' \Big\|_{\Kato^*_x L^{\infty}_t} &\les \|F\|_{L^1_{x, t}},\\
\Big\|\int_{t'<t} \frac{\cos((t-t')\sqrt {H}) P_c}{H} F(t') \dd 
t' \Big\|_{L^{\infty}_{x, t}} &\les \|F\|_{\Kato_x L^1_t}.
\end{aligned}$$
\end{theorem}
$\Kato^*$ can be replaced by $L^{3, \infty}$ (weak-$L^3$) wherever it appears in 
the statement of Theorem \ref{main_theorem}. Also note that $W^{1, 1} \subset 
\Kato$ by Lemma \ref{kato_w11}.

By interpolation we obtain a wider family of inequalities:
\begin{corollary}\lb{cor_interpolation} Consider a real-valued potential $V \in 
\Kato_0$ on $\R^3$ such that $-\Delta+V$ has 
no eigenvalues on $[0, \infty)$ and no resonance at zero.
Then, for $0 \leq \theta \leq 1$
\begin{align}
\Big\|\int_{t'<t}\frac {\sin((t-t') \sqrt H) P_c}{\sqrt H} F(t') 
\dd t'\Big\|_{(\Kato^{1-\theta})^*_x L^p_t} &\les \|F\|_{K^{\theta}_x L^p_t} 
\label{eq:sinKato}\\
\Big\|\int_{t'<t}\frac{\cos((t-t') \sqrt H) P_c}{H} F(t') \dd 
t'\Big\|_{(\Kato^{1-\theta})^*_x L^{\infty}_t} &\les \|F\|_{\Kato^\theta_x 
L^1_t}. \notag
\end{align}
More generally for any $0 \leq \theta_1, \theta_2 \leq 1$ and 
$1 \leq p \leq q \leq \infty$ with $\theta_1 + \theta_2 \leq 1$ and 
$\frac{1}{p} - \frac{1}{q} = \theta_1 + \theta_2$,
\begin{equation} \label{eq:cosKato}
\Bignorm[\int_{t'<t}\frac{\cos((t-t') \sqrt H) P_c}{H} F(t') 
\dd t'][(\Kato^{\theta_2})^*_xL^q_t] \les
\norm[F][\Kato^{\theta_1}_xL^p_t].
\end{equation}
Additionally, for $\theta_1, \theta_2$ as above and 
$\frac{1}{q} = 1 - \theta_1 - \theta_2$,
\begin{align}
\Bignorm[\frac{\sin(t\sqrt{H})P_c}{\sqrt{H}} f][(\Kato^{\theta_2})^*_xL^q_t]
& \les \norm[\nabla f][\Kato^{\theta_1}], \notag \\
\bignorm[\cos(t\sqrt{H})P_c f][(\Kato^{\theta_2})^*_x L^q_t]
&\les \norm[\Delta f][\Kato^{\theta_1}], \notag \\ 
\Bignorm[\frac{\sin(t\sqrt{H})P_c}{\sqrt{H}} f][(\Kato^{\theta_2})^*_x] 
& \les |t|^{\theta_1 + \theta_2 - 1}\norm[\nabla f][\Kato^{\theta_1}], \notag\\
\big\|\cos(t \sqrt H) P_c f\big\|_{(\Kato^{\theta_2})^*_x} 
&\les |t|^{\theta_1 + \theta_2 - 1} 
\norm[\Delta f][K^{\theta_1}].  \label{eq:dispKato}
\end{align}
The following Lorentz space inequalities are also valid in the range
$1 < p,q < \infty$, $1 \leq s \leq \infty$:
\begin{equation} \label{eq:sinLorentz}
\Bignorm[\int_{t'<t} \frac {\sin((t-t') \sqrt H) P_c}{\sqrt H} F(t') 
\dd t'][L^{q, s}_x L^r_t] \leq C_{pq} \norm[F][L^{p, s}_x L^r_t]
\end{equation}
for $\frac{1}{p} - \frac{1}{q} = \frac{2}{3}$ and $1 \leq r \leq \infty$,
\begin{equation} \label{eq:cosLorentz}
\Big\|\int_{t'<t}\frac{\cos((t-t') \sqrt H) P_c}{H} F(t') \dd t'
\Big\|_{L^{q, s}_x L^r_t} \leq C_{pq} \|F\|_{L^{p, s}_x L^{\tilde{r}}_t},
\end{equation}
for $\frac 1 p - \frac 1 q = 1-\frac13(\frac{1}{\tilde{r}} - \frac{1}{r})$, 
$1 \leq \tilde{r} < r \leq \infty$, and
\begin{equation} \label{eq:dispLorentz}
\Bignorm[\frac{\cos(t\sqrt{H})P_c}{H} f][L^{q,s}] \leq C_{pq} |t|^{-1/r}
\norm[f][L^{p,s}]
\end{equation}
for $\frac{1}{p} - \frac{1}{q} = \frac{2}{3} + \frac{1}{3r}$,
$1 < r \leq \infty$, $1<p, q<\infty$.  

Under the stronger assumption $V \in L^{3/2,1}(\R^3)$ and the same $(p,q,r)$
as in~\eqref{eq:dispLorentz}, these inequalities hold as well:
\begin{align*}
\|\cos(t \sqrt H) P_c f\|_{L^{q, s}_x L^r_t} &\leq C_{pq} \|\Delta 
f\|_{L^{p,s}}, \\
\norm[\cos(t \sqrt H) P_c f][L^{q,s}] &\leq C_{pq} |t|^{-1/r} \norm[\Delta 
f][L^{p,s}].
\end{align*}
The constant $C_{pq}$ satisfies the bound
$C_{pq} \les \frac{(p'+q)^2}{p'q}$, where $p' := \frac{p}{p-1}$.
In particular, $C_{pq}$ is uniformly bounded over all pairs
where $p' = q$.
\end{corollary}

Finally, we state our results concerning homogeneous Strichartz estimates. 
Several of the endpoint bounds are valid, and one
of them doubles as a statement about maximal operators:
\begin{theorem}[Strichartz and reversed-norm Strichartz 
estimates]\lb{thm_maximal}
Consider a 
real-valued potential $V \in \Kato_0$ on $\R^3$ such that $-\Delta+V$ has 
no eigenvalues on $[0, \infty)$ and no resonance at zero.  Then for all
$0 \leq s <\frac 32$ and $\frac{1}{r} + \frac{3}{q} = \frac32 - s$ in the range
$0 \leq \frac{1}{r} < \frac{s}{2}$,
\begin{equation} \label{eq:Strichartz1}
\Bignorm[\frac{e^{it\sqrt{H}}P_c}{H^{s/2}} f][L^r_t L^q_x] \les
\norm[f][2].
\end{equation}
The initial-value problem for the wave equation~\eqref{eq:wave}
satisfies
\begin{equation} \label{eq:Strichartz}
\Bignorm[\cos(t\sqrt{H})P_c f_0 + \frac{\sin(t\sqrt{H})P_c}{\sqrt{H}} f_1][
L^r_t L^q_x]
\les  \norm[f_0][\dot{H}^s] + \norm[f_1][\dot{H}^{s-1}],
\end{equation}
for all $s \in [0,1]$, and also for $1 < s < \frac32$ under the additional
assumption $V \in L^{3/2, 1}$.
The case $r=\infty$ is included for all $0\leq s<3/2$ and the case $q=\infty$
for $1<s<\frac 32$.

For $\frac 12<s<\frac 32$, the following reversed-norm Strichartz
inequalities are also valid:  For all
$\theta$, $r$ with $\max(0, 1-s) \leq \theta \leq \min(\frac1 2, \frac 3 2 - s)$ 
and $\theta + 
\frac{1}{r} = \frac32-s$,
\be \lb{eq:StrichKato}
\Big\|\frac{e^{it\sqrt H} P_c}{H^{s/2}} f\Big\|_{(\Kato^{\theta})^*_x L^{r, 
2}_t} 
\les \|f\|_{L^2},\\
\ee 
which includes the endpoint case $L^\infty_x L^{\frac 2 {3-2s}, 2}_t$ for $1\leq 
s<\frac 32$ and $\theta = 0$. Furthermore
\be \lb{maximal}
\Big\|\frac{e^{it\sqrt H} P_c}{H^{s/2}} f\Big\|_{L^{q,2}_x L^r_t}  \les  
\|f\|_{L^2}
\ee
for all pairs $(q,r)$ with $\max(6, \frac 6{3-2s}) \leq q \leq \frac 3{\max(1-s, 
0)}$, $q \ne \infty$, and $\frac{3}{q} + \frac{1}{r} = \frac32-s$. 
This includes the endpoint $L^{\frac 6{3-2s}, 2}_x L^\infty_t$ when $1 \leq s < 
\frac 32$.

Solutions to the homogeneous wave equation~\eqref{eq:wave} satisfy
\begin{equation} \label{eq:StrichSobolev}
\Bignorm[\cos(t\sqrt{H})P_c f_0 + \frac{\sin(t\sqrt{H})P_c}{\sqrt{H}} f_1][
(\Kato^\theta)^*_xL^r_t \cap L^{q,2}_x L^r_t]
\les  \norm[f_0][\dot{H}^s] + \norm[f_1][\dot H^{s-1}]
\end{equation}
for $s \in [0,1]$ and also for $1 < s < \frac32$ under the additional assumption
$V \in L^{3/2, 1}$.
\end{theorem}
In several non-endpoint estimates, $L^p$ spaces can be strengthened to $L^{p, 
2}$, as shown in the proof.

\begin{remark}
There is a Kato-class version of~\eqref{eq:Strichartz} analogous
to~\eqref{eq:StrichKato}.  These estimates are described in terms of
 Schechter's spaces $M_{3-2s\theta, s+1}$ (as defined in~\cite{Si82}), which
form a complex interpolation family between $\Kato^\theta$ and $L^2(\R^3)$.
\end{remark}

\begin{remark}
Theorem~\ref{thm_maximal} is presented as a consequence of
Corollary~\ref{cor_interpolation}; it is not a complete list of valid
Strichartz and  reversed-norm Strichartz inequalities.  
A self-contained proof of the ``sharp-admissible" case $r=\frac{2}{s}$, $s < 1$
in~\eqref{eq:Strichartz} likely demands several extra steps (e.g. 
Littlewood-Paley decomposition) to make $L^\infty$ safe for complex
interpolation.
\end{remark}

On the other hand, the main technical step (Theorem~\ref{thm:goal_lattice})
makes it possible to transfer reversed-norm estimates from the Laplacian to $H$
regardless of their initial proof.  This allows us to confirm Strichartz 
inequalities
over half of the sharp-admissible range.

\begin{theorem} \label{thm:sharpStrichartz}
Consider a real-valued potential $V \in L^{3/2,1}(\R^3)$ such that
$-\Delta + V$ has no eigenvalues on $[0,\infty)$ and no resonances at zero.
then for all $0 \leq s \leq \frac12$,
\begin{align} \label{eq:sharpStrichartz1}
\Bignorm[\frac{e^{it\sqrt{H}}P_c}{H^{s/2}}][L^{2/s}_tL^{2/(1-s)}_x]
&\les \norm[f][2] \qquad {\textit and} \\
\Bignorm[\cos(t\sqrt{H})P_c f_0 + \frac{\sin(t\sqrt{H})P_c}{\sqrt{H}} f_1][
L^{2/s}_tL^{2/(1-s)}_x]
&\les  \norm[f_0][\dot{H}^s] + \norm[f_1][\dot H^{s-1}]. 
\label{eq:sharpStrichartz}
\end{align}
\end{theorem}

The $L^{\frac 6 {3-2s}, 2}_x L^\infty_t$ endpoint
in~\eqref{eq:StrichSobolev} is a bound on the maximal function 
for the wave 
propagator, with the following consequence:
\begin{theorem}\label{limit}
Consider a real-valued potential $V \in \Kato_0$ on $\R^3$ such that $-\Delta+V$ 
has no eigenvalues on $[0,\infty)$ and no resonance at zero. Then 
for $f_0 \in \dot{H}^1$ and $f_1 \in L^2$
\begin{equation} \label{max}
\lim_{t\to 0} \Big(\cos(t\sqrt{H})f_0 + 
\frac{\sin(t\sqrt{H})}{\sqrt{H}}f_1\Big)(x) = f_0(x)
\end{equation}
at almost every $x \in \R^3$.
\end{theorem}
More generally, such results hold for $\dot H^s$, $1 \leq s < 3/2$.

\subsection{History of the problem}
Strichartz estimates are a fundamental tool in the study of the wave equation. 
Other inequalities employed in its study include local energy decay, Morawetz 
estimates \cite{morawetz}, weighted Keel--Smith--Sogge-type estimates 
\cite{KSS}, and Killing-field-based, i.e.\ Klainerman--Sobolev \cite{Kla}, 
estimates. These techniques are not mutually exclusive and are often used 
together or even in combination.

Strichartz-type estimates only hold for the dispersive part of the evolution, 
corresponding to the projection on the continuous spectrum. The optimal rate of 
decay requires the absence of zero eigenvalues and resonances, i.e.\ of nonzero 
solutions of the equation $H\Psi  = 0$. Eigenvalues and resonances at zero, even 
when specifically excluded by means of a spectral projection, nevertheless lead 
to a lower rate of decay for the dispersive part of the evolution.

In the case of Schr\"{o}dinger's equation in one dimension, Kato-type smoothing 
estimates hold in a similar norm:
$$
\||\dl|^{1/2} e^{it\Delta} f\|_{L^{\infty}_x L^2_t} \les \|f\|_2.
$$
Such estimates have also been obtained previously by Kenig--Ponce--Vega 
\cite{kpv1}, \cite{kpv2} in one dimension for the Korteweg-deVries and Benjamin--Ono equations. Reversed-norm estimates are new in higher dimensions for the 
wave equation; see \cite{lizhang} for a related result for Schr\"{o}dinger's 
equation with radial data.

Strichartz estimates for the wave equation have a long history, going back to 
the works of Strichartz \cite{strichartz} and Segal \cite{segal} in $\R^2$. 
Previously-known Strichartz inequalities for the free wave equation (without 
potential) belong to Georgiev--Karadzhov--Visciglia \cite{georgiev}, Harmse 
\cite{harmse}, Kapitanski \cite{kapitanski}, Ginebre--Velo \cite{ginebre1} 
\cite{ginebre}, Oberlin \cite{oberlin}, Lindblad--Sogge \cite{liso}, and 
Nakamura--Ozawa \cite{nakamura}. In other settings we cite the work of 
Mockenhaupt--Seeger--Sogge \cite{mss}.

Keel--Tao \cite{tao} also obtained sharp Strichartz estimates for the free wave 
equation in $\R^4$ and higher dimensions --- and everything except the endpoint 
in $\R^3$. These estimates were further extended by Foschi \cite{foschi}, in the 
inhomogeneous case.

In this paper we seek to prove similar estimates for the cosine and sine 
evolutions
$\frac {\sin(t\sqrt H)P_c}{\sqrt H}$ and $\cos(t\sqrt H) P_c$ related to the 
perturbed Hamiltonian $H= {-\Delta} + V$.

Many such results were initially proved only for the free wave equation. 
However, the boundedness of wave operators shown by Yajima in \cite{yajima} 
implies that results in the free case can be extended to the case of a perturbed 
Hamiltonian $H={-\Delta}+V$, if the potential has sufficient decay: $|V(x)| \les 
\langle x \rangle^{-6-\epsilon}$ in $\R^3$. Still, Yajima's method does not 
apply to the results of Theorem \ref{main_theorem}, because the $L^p$ 
boundedness of the wave operators does not address the issue of reversed 
spacetime norms.

Strichartz estimates for the wave equation in $\R^d$ with critical potentials 
were obtained by Burq--Planchon--Stalker--Tahvildar-Zadeh \cite{burq}, 
\cite{burq2} and by D'Ancona--Pierfelice \cite{DaPi}.
The former results apply to potentials an with inverse-square pointwise bound,
$|V(x)| \les |x|^{-2}$, the primary examples of which lie outside $\Kato_0$.
Strichartz estimates are shown to fail in this larger class if there is no control over
the negative part of the potential; some additional smoothness in the
radial direction (where $x=r\omega$) is also assumed.  The dispersive
and Strichartz estimates in \cite{DaPi} apply
to to all $V \in \Kato_0$ whose negative part satisfies $\norm[V^-][\Kato] < 2\pi$.
In this paper the size bound on $V^-$ is replaced with the weaker hypothesis that
the operator $-\Delta + V$ has no eigenvalues or resonances along its essential
spectrum.

The current paper's results are also comparable to those of Rogers--Villarroya. 
In \cite{rogers}, these authors studied maximal operators for the wave equation 
in $\R^d$, defined by
$$
M_1 f(x) := \sup_{t \in [0, 1]} |(e^{it\sqrt \Delta} f)(x)| \text{ and } M f(x) 
:= 
\sup_{t \in \R} |(e^{it\sqrt{{-\Delta}}} f)(x)|.
$$
In particular they showed estimates of the form
$$
\|M f\|_{L^q} \les \|f\|_{H^s},
$$
where $q \geq \frac {2(d+1)}{d-1}$ and $s>d(1/2-1/q)$. The estimate 
(\ref{maximal}) implies that for $1 \leq s < \frac 32$
$$
\|\sup_t |(e^{it \sqrt H} P_c f)(x)|\|_{L^{\frac 6{3-2s}, 2}_x} \les \|f\|_{\dot 
H^s_x}.
$$
This corresponds to an endpoint not treated in \cite{rogers}. We also prove this 
inequality for the general case of $H={-\Delta}+V$, in addition to the free wave 
equation.

Some of the results obtained here for $\K_0$ potentials can be extended to more 
general Kato-class potentials, including singular measures;
see, in this direction, the results of \cite{michael}.

Many papers are dedicated to Strichartz estimates for Schr\"{o}dinger's 
equation, whose proof is largely similar to that of Strichartz estimates for the 
wave equation.
Indeed, the key step of using an operator-valued Wiener $L^1$-inversion theorem
was originally applied to
Schr\"{o}dinger's equation by the authors in \cite{bec}, \cite{becgol}, 
and \cite{michael}.  The abstract Wiener theorem presented in
Section~\ref{sec:Wiener} is borrowed largely intact from these works,
however the wave equation encourages it to appear in several new disguises
including a weighted-$L^1$ version (Proposition~\ref{thm:2ndgoal})
not previously considered.

\subsection{Nonlinear applications}
\begin{proposition}
For $V \in L^{3/2, 1}$, consider the energy-critical equation
$$
\partial^2_t f - \Delta f + V f = F \pm f^5, (f(0), \partial_t f(0)) = (f_0, 
f_1).
$$
Assume that $H={-\Delta}+V$ has neither eigenvalues, nor resonances. Then for 
sufficiently small $(f_0, f_1) \in \dot H^1 \times L^2$ and $F \in L^{6/5, 2}_x 
L^\infty_t$, there exists a unique solution $f \in L^{6, 2}_x L^{\infty}_t$.

When $(f_0, f_1) \in \dot H^1 \times L^2$ and $F \in L^{6/5, 2}_x 
L^\infty_t \cap L^{3/2, 1}_x L^2_t$, the solution $f$ is in $L^{6, 2}_x 
L^{\infty}_t \cap L^{\infty}_x L^2_t$.
\end{proposition}

By the conservation of energy, in both cases $(f, \partial_t f) \in L^{\infty}_t 
\dot H^1_x \times L^{\infty}_t L^2_x$.

Similar results hold for $\dot H^s$-critical wave equations with $1 \leq s < 
3/2$.

\begin{proof} By Strichartz estimates (Theorem \ref{thm_maximal}) we obtain
$$
\|\cos(t \sqrt H) P_c f\|_{L^{\infty}_{x} L^2_t \cap L^{6, 2}_x L^\infty_t} \les 
\|f\|_{\dot H^1}.
$$
Likewise,
$$
\Big\|\frac {\sin(t \sqrt H) P_c}{\sqrt H} f(x)\Big\|_{L^{\infty}_x L^2_t \cap 
L^{6, 2}_x L^\infty_t} \les 
\|f\|_2
$$
and by Theorem \ref{main_theorem} and Corollary \ref{cor_interpolation}
$$\begin{aligned}
\Big\|\int_{s<t} \frac {\sin((t-s) \sqrt H) P_c}{\sqrt H} F(x, s) \dd 
s\Big\|_{L^{\infty}_x L^2_t} &\les \|F\|_{\Kato_x L^2_t} \les \|F\|_{L^{3/2, 
1}_x L^2_t},\\
\Big\|\int_{s<t} \frac {\sin((t-s) \sqrt H) P_c}{\sqrt H} F(x, s) \dd 
s\Big\|_{L^{6, 2}_x L^\infty_t} &\les \|F\|_{L^{6/5, 2}_x L^\infty_t}.
\end{aligned}$$
We retrieve the solution by means of a fixed-point scheme, based on the 
Strichartz inequalities above. Note that
$$
\|f^5\|_{L^{6/5, 2}_x L^\infty_t} \leq \|f\|^5_{L^{6, 2}_x L^\infty_t}
$$
and
$$
\|f^5\|_{L^{3/2, 1}_x L^2_t} \leq \|f\|^4_{L^{6, 2}_x L^\infty_t} 
\|f\|_{L^{\infty}_x L^2_t}.
$$
\end{proof}

\section{Definitions and basic estimates}

The global Kato norm, which defines our admissible class of potentials $V(x)$, 
is highly compatible with
Sobolev spaces based on $L^1(\R^3)$ or $L^\infty(\R^3)$.  Let $\dot{W}^{1,1}$ 
indicate the completion of
compactly supported test functions under the norm $\norm[f][\dot{W}^{1,1}] := 
\norm[\nabla f][1]$.

\begin{lemma}\lb{kato_w11} If $f \in \dot{W}^{1,1}$  then $f \in \Kato_0$ and 
$\|f\|_{\K} \leq \|f\|_{\dot W^{1, 1}}$.
\end{lemma}
\begin{proof}
Suppose $f \in C^1_c(\R^3)$.
Fix any point $y$, and then write the integral for the Kato norm in polar 
coordinates $x = y + r\omega$.
$$\begin{aligned}
\int_{\R^3} \frac {|f(x)|}{|x-y|} \dd x &= \int_0^{\infty} \int_{S^2} 
|f(y+r\omega)| \dd \omega\; r \dd r = \\
&= -\int_0^{\infty} \int_{S^2} \partial_r |f(y+r\omega)| \dd \omega \frac {r^2} 
2 \dd r  \les \|f\|_{\dot W^{1, 1}}.
\end{aligned}$$
since $\big|\partial_r |f(y+r\omega)|\big| = |\partial_r f(y+r\omega)|$ almost 
everywhere.  Take the supremum over
$y \in \R^3$ to obtain the same bound for $\norm[f][\Kato]$, and limits to 
extend the result to all $f \in \dot{W}^{1,1}$.
\end{proof}

The next lemma shows that functions with two derivatives in $\Kato$ and 
some decay at infinity are in fact bounded.
\begin{lemma}\lb{kato_linfty} Assume that $D^2 f \in \Kato$ and $f$ can be 
approximated by smooth functions of compact support. Then $\|f\|_{L^{\infty}} 
\les \|D^2 f\|_{\Kato}$.
\end{lemma}
 \begin{proof}
Note that for bounded functions $f$ of compact support, for every $x$
$$
\lim_{r \to \infty} \int_{S^2} f(x+r\omega) \dd \omega = 0.
$$
Then
$$\begin{aligned}
f(x) &= - \frac 1 {4\pi} \int_{S^2} \int_0^{\infty} \frac {\partial 
f}{\partial_r}(x+r\omega) \dd r \dd \omega \\
&= \frac 1 {4\pi} \int_{S^2} \int_0^{\infty} \int_r^{\infty} \frac {\partial^2 
f}{\partial_r^2}(x+s\omega) \dd s \dd r \dd \omega \\
&= \frac 1 {4\pi} \int_{S^2} \int_0^\infty s \frac {\partial^2 
f}{\partial_r^2}(x+s\omega) \dd s \dd \omega \\
&\les \int_{\R^3} \frac {|D^2 f(y)|}{|x-y|} \dd y \les \|D^2 f\|_{\Kato}.
\end{aligned}$$
By approximation we infer the result in the general case.
\end{proof}

Finally, we show the following inclusion between Kato-type spaces:
\begin{lemma}\lb{kato_sob} If $f$ is approximable by bounded compactly supported 
functions
$$
\sup_x \int_\R \frac {|f(y)| \dd y}{|x-y|^2}  \les \|\dl f\|_{\Kato}.
$$
\end{lemma}
\begin{proof} Take $f$ bounded and compactly supported. Then for every $x$
$$
\lim_{r \to \infty} \int_{S^2} |f(x+r\omega)| r \dd \omega = 0.
$$
Consequently
$$
\int_\R \int_{S^2} |f(x+r\omega)| \dd \omega \dd r = - \int_\R \int_{S^2} 
\partial_r |f(x+r\omega)| r \dd \omega \dd r \les \|\dl f\|_{\Kato}
$$
because $|\partial_r |f(x+r\omega)|| = |\partial_r f(x+r\omega)|$ almost 
everywhere.

The conclusion then follows by approximation.
\end{proof}

A straightforward approximation argument shows that every potential $V \in 
\Kato_0$ satisfies the local Kato condition 
\begin{equation} \label{eq:localKato}
\lim_{\epsilon\to 0} \sup_{y\in\R^3} \int_{|x-y| < \epsilon} \frac{|V(x)|}{|x-
y|}\,dx
 = 0
\end{equation}
and the distal Kato condition
\begin{equation} \label{eq:distalKato}
\lim_{R \to \infty} \sup_{y\in\R^3} \int_{|x-y| > R} \frac{|V(x)|}{|x-
y|}\,dx
 = 0.
\end{equation}
These properties imply that
$H = {-\Delta} + V$ is essentially self-adjoint with spectrum bounded
below by $-M$ for some $M < \infty$~\cite{Si82}.
With the assumption that zero is neither an eigenvalue nor a resonance, there 
are at most finitely many negative eigenvalues $\lambda_j$ with corresponding 
orthogonal spectral projections $P_j$. We denote by $P_c := I - \sum_j P_j$ the 
projection on the absolutely continuous spectrum. 

For each $z \in \Compl \setminus \R^+$
define $R_0(z) :=({-\Delta}-z)^{-1}$.
The operators $R_0(z)$ are all bounded on $L^2(\R^3)$ and
act explicitly by convolution with the kernel
\begin{equation*}
R_0(z)(x,y) = \frac{e^{i\sqrt{z}|x|}}{4\pi|x|},
\end{equation*}
where $\sqrt{z}$ is taken to have positive imaginary part.  
On the boundary $\lambda \in \R^+$ the resolvents $R_0(\lambda \pm i0)$ are 
defined as  $\lim_{\eps\downarrow 0} 
R_0(\lambda\pm i\eps)$. These operators are not $L^2$-bounded but instead 
satisfy a uniform mapping estimate
from $L^{6/5, 2}$ to $L^{6, 2}$.

The perturbed resolvent $R_V(z) := (H-z)^{-1}$ does not have a closed-form 
formula for its integral kernel,
however it can be expressed in terms of $R_0(z)$ via the identity
\begin{equation} \label{eq:ResIdent}
R_V(z) = (I + R_0(z)V)^{-1}R_0(z) = R_0(z)(I + VR_0(z))^{-1}. 
\end{equation}
Boundary values $R_V^\pm(\lambda)$ along the positive real axis are once again 
defined
by continuation, that is $R_V^\pm(\lambda) = \lim_{\eps\downarrow0} R_V(\lambda 
\pm i\eps)$.
The two continuations do not coincide; the difference between them is (up to a 
constant factor) the
absolutely continuous spectral measure of $H$.

Returning to the formal solution~\eqref{eq:formalsolution} of the wave equation 
with a potential,
we wish to analyze the component of the solution orthogonal to eigenfunctions of 
$H$, that is
$$
P_c f(t) = \cos(t \sqrt H) P_c f_0 + \frac {\sin(t \sqrt H) P_c}{\sqrt H} g_0 + 
\int_0^t \frac {\sin((t-s)\sqrt H) P_c}{\sqrt H} F(s) \, ds,
$$
where $\cos(t \sqrt H)P_c$ and $\frac {\sin(t \sqrt H) P_c}{\sqrt H}$ are 
defined by means of operator calculus:
\begin{align*}
\cos(t \sqrt H) P_c f &:= \frac{1}{2\pi i}\int_0^\infty \cos(t \sqrt \lambda)
[R_V^+(\lambda)-R_V^-(\lambda)]f \,d\lambda, \\
\frac {\sin(t \sqrt H) P_c}{\sqrt H} f &:= \frac{1}{2\pi i}\int_0^\infty \frac 
{\sin(t \sqrt \lambda)}{\sqrt \lambda}
[R_V^+(\lambda)-R_V^-(\lambda)]f \,d\lambda.
\end{align*}
This definition makes $\cos(t \sqrt H)P_c$ and $\frac {\sin(t \sqrt H) 
P_c}{\sqrt H}$ $L^2_x$-bounded for every $t$.
So long as zero is a regular point of the spectrum, it is possible to effect a 
change of variables
$\lambda \mapsto \lambda^2$ leading to the identities
\begin{align} 
\cos(t\sqrt{H})P_c f &= \frac{1}{\pi i}\int_{-\infty}^\infty \lambda 
\cos(t\lambda) R_V^+(\lambda^2) f \,d\lambda \label{eq:Stonecosine}\\
\frac{\sin(t\sqrt{H})P_c}{\sqrt{H}}f &= \frac{1}{\pi i}\int_{-\infty}^\infty 
\sin(t\lambda) R_V^+(\lambda^2) f\,d\lambda. \label{eq:Stonesine}
\end{align}
To be precise, $R_V^+(\lambda^2)$ denotes 
$\lim_{\eps\downarrow0}R_V((\lambda+i\eps)^2)$, which coincides with
$R_V^\pm(\lambda^2)$ according to the sign of $\lambda$.  Presented this way, we 
see that the linear propagators
for the perturbed wave equation are closely tied to the Fourier transform (in 
$\lambda$) of the resolvent
$R_V^+(\lambda^2) = R_0^+(\lambda^2)(I + VR_0^+(\lambda^2))^{-1}$.
Indeed, the principal effort behind 
Theorem~\ref{main_theorem} is to establish
useful bounds on the Fourier transform of the perturbative factor $(I + 
VR_0^+(\lambda^2))^{-1}$.
We obtain these bounds as an application of the abstract Wiener inversion 
theorem in Section~\ref{sec:Wiener}.

There is no material change to these estimates if one considers the resolvent 
continuation from below
(i.e. $R_0^-(\lambda^2)$) or applies inverse Fourier transforms instead, as has 
been done in previous studies.
Keeping the notation in~\cite{becgol}, let 
\begin{equation} \label{eq:T_hat}
\hat{T}^\pm(\lambda,x,y) := VR_0^\pm(\lambda^2)(x,y)
 = V(x) \frac{e^{\pm i\lambda|x-y|}}{4\pi |x-y|}.
\end{equation}
Its inverse Fourier transform in the $\lambda$ variable has the explicit form
\begin{equation} \label{eq:T}
\begin{aligned}
T^\pm(\rho,x,y) &= (4\pi|x-y|)^{-1}V(x)\delta_{\mp|x-y|}(\rho) \\
&= (\mp 4\pi \rho)^{-1} V(x) \delta_{\mp|x-y|}(\rho),
\end{aligned}
\end{equation}
which is the restriction of $V(x)/(4\pi |x-y|)$ to the light cone
$|x-y| = \pm \rho$.

We will return to these operators in Section~\ref{sec:proof} with the goal of
understanding properties of $(I + \hat{T}^\pm(\lambda))^{-1}$.
In Section~\ref{sec:Wiener}, where the techniques for constructing an operator 
inverse
are developed, the family of operators $\hat{T}(\lambda)$ will be of a more 
general character.

\section{An Abstract Wiener Theorem} \label{sec:Wiener}

Let $X$ be a Banach lattice of complex-valued functions over a measure space 
$(\X,\mu)$.
The main lattice properties are that $f \mapsto |f|$ is an isometry in $X$, and 
if $f \in X$ and 
$|g| \leq f$, then $g \in X$ with lesser norm.  Relevant examples include 
$L^p(\mu)$, 
$1 \leq p \leq \infty$, $\Kato$ and $\Kato^*$ as defined in 
Section~\ref{sec:intro}, and their
interpolants. The space of finite complex-valued measures on $\R$, which we 
denote by $\M$,
is also a Banach lattice.

The compound space $X_x\M_\rho$ then consists of measures $\nu$ on $\R \times 
\X$ for which $M(x) = \norm[\nu(\,\cdot\, , x)][\M]$ is finite $\mu$-almost 
everywhere and
belongs to $X$.

\begin{definition}
Let $\U_X$ be the set of bounded operators from $X_x \M_{\rho}$ to itself 
defined formally by
the integral
\be\lb{u_x}
(T F)(\rho, x) := \int_{-\infty}^{\infty} \int_{\X} T(\rho-\sigma, x, y) 
F(\sigma, y) \,\mu(dy) \dd \sigma,
\ee
and possessing the property that
$|T(\rho-\sigma, x, y)|$ also produces a bounded integral operator on 
$X_x\M_\rho$.
In the general case, $T(\rho, x,y)$ may be a measure on $\R \times \X \times \X$ 
and its associated
linear operator is defined weakly by pairing with a second function in 
$X_x^*B(\R)_\rho$.

Setting the norm $\norm[T][\U_X] := \norm[|T|][\B(X_x\M_\rho)]$ gives $\U_X$
the structure of a Banach lattice.  In addition, $\U_X$ is a Banach algebra 
under
the natural composition of operators on $X_x\M_\rho$.
\end{definition}

Given an element $T \in \U_X$, let $M(T)$ be the marginal distribution of $|T|$ 
on $\X \times \X$,
which can be written formally as
\be\lb{mt}
M(T)(x,y) =  \int_{-\infty}^{\infty} |T(\rho, x, y)|  \dd \rho.
\ee
It follows from the construction that $\norm[T][\B(X_x\M_\rho)] \leq 
\norm[M(T)][\B(X)] = \norm[T][\U_X]$.  

Kernels $T \in \U_X$ also define bounded operators from $X$ to $X_x \M_{\rho}$ 
by
\be\lb{image}
(T f)(\rho, x) = \int_{-\infty}^{\infty} \int_{\X} T(\rho, x, y) f(y) \,\mu(dy).
\ee
For many of the statements in Theorem~\ref{main_theorem}, it will be significant 
that $T \in \U_X$
acts as a bounded operator on $X_x L^p_\rho$ as well for any $1 \leq p \leq 
\infty$.  This is a
consequence of~\eqref{u_x} and~\eqref{mt} and the fact that convolution with a 
finite measure
preserves $L^p(\R)$.  Finally, each element $T \in \U_X$ has an adjoint operator
$T^*(\rho,x, y) = \overline{T(-\rho, y, x)}$ which belongs to $\U_{X^*}$ because
$M(T^*) = M(T)^*$.

We take the Fourier transform of $T \in \U_X$, $\hat T(\lambda) \in \B(X)$, to 
be the operator with
kernel
\be\lb{fourier}
\hat T(\lambda, x, y) = \int_{-\infty}^{\infty} e^{-i\rho\lambda} T(\rho, x, y) 
\dd \rho,
\ee
if $T(\rho,x,y)$ is a function, 
or more generally the marginal distribution of $e^{-i\rho\lambda}T$.
It satisfies the standard product relation $(ST)^{\wedge}(\lambda) = \hat 
S(\lambda) \hat T(\lambda)$.  
For each $\lambda \in \R$, $\hat{T}(\lambda)$ is dominated pointwise by $M(T)$, 
hence
$\|\hat T(\lambda)\|_{\B(X)} \leq \|T\|_{\U_X}$.

There is an identity element $\1 \in \U_X$ whose integral kernel $\1(\rho,x,y)$ 
is the lifting of $\delta_x(y)$ onto the 
"centered" diagonal $\{\rho = 0\} \times \{x = y\} \subset \R\times\X\times\X$ 
via the natural identification with points
in $X$.  One can verify that $M(\1)$ is the identity operator on $X$, and that 
the Fourier transform of $\mb 1$ is $\hat{\mb 1}(\lambda) = I_X$ for each 
$\lambda \in \R$.

\begin{proposition} \label{thm:Wiener_2}
Suppose $T \in \U_X$ is such that
\begin{enumerate}
\item[C1] For some $N > 0$, $\lim\limits_{\delta \to 0} 
\norm[T^N(\rho,x,y) - T^N(\rho-\delta,x,y)][\U_X] = 0$.
\item[C2] $\lim\limits_{R \to \infty}
\norm[\chi_{|\rho| \ge R} T][\U_X] = 0$.
\end{enumerate}
If $I + \hat{T}(\lambda)$ is an invertible element of $\B(X)$ for every
$\lambda \in \R$, then ${\mb 1} + T$ is invertible in 
$\U_X$.
\end{proposition}
The argument used in \cite{becgol}, used for a slightly different algebra based 
on $L^1_\rho X_x$ instead of 
$X_x\M_\rho$, can be repeated without substantial modification. We reproduce it 
below for the reader's convenience.

\begin{proof}
It suffices to show that $(I + \hat{T}(\lambda))^{-1}$ is the Fourier transform
of an element  $S \in \U_X$. 
Let $\eta: \R \to \R$ be a standard
cutoff function.  For any $L \in \R$, the large-$\lambda$ restriction
$(1 - \eta(\lambda/L))\hat{T}(\lambda)$ is the Fourier transform of
\begin{equation*}
S_L(\rho) = \big(T - L\check{\eta}(L\,\cdot\,) * T\big)(\rho) = 
\int_\R L \check{\eta}(L\sigma) [T(\rho) - T(\rho-\sigma)]\,d\sigma.
\end{equation*}
Dependence on $(x,y)$ in the kernels $T(\rho,x,y)$ is suppressed in the 
expression above.
If condition~C1 is satisfied with $N = 1$, then 
the $\U_X$ norm of the right-hand side vanishes as $L \to \infty$.
Thus there is a fixed number $L$ so that $\sum_{k=0}^\infty (-S_L)^k$ is a 
convergent
series in $\U_X$.  Its Fourier transform is 
\begin{equation*}
\sum_{k=0}^\infty (-1)^k \Big(\big(1 - \eta(\lambda/L))\hat{T}(\lambda)\Big)^k
= \Big(I + \big(1-\eta(\lambda/L)\big)\hat{T}(\lambda)\Big)^{-1}
\end{equation*}
which agrees with $(I + \hat{T}(\lambda))^{-1}$ for all $\lambda > 2L$.  Apply 
the decomposition
\begin{align*}
(I + \hat{T}(\lambda))^{-1} &= \eta(\lambda/2L) (I + \hat{T}(\lambda))^{-1} + 
\big(1-\eta(\lambda/2L)\big) (I + \hat{T}(\lambda))^{-1} \\
&= \eta(\lambda/2L) (I + \hat{T}(\lambda))^{-1} + (1 - \eta(\lambda/2L)) 
\big(I + (1-\eta(\lambda/L))\hat{T}(\lambda)\big)^{-1}
\end{align*}
If instead $T^N$ satisfies C1 then one observes that
\begin{equation*}
(1 - \eta(\lambda/2L))\big(I + \hat{T}(\lambda)\big)^{-1}
  = (1 - \eta(\lambda/2L))\big(I - (-\hat{T}(\lambda))^N \big)^{-1}
\sum_{k=0}^{N-1}(-1)^k \hat{T}^k(\lambda).
\end{equation*}

We then construct a local inverse for $I + \hat{T}(\lambda)$ in the neighborhood 
of
any $\lambda_0 \in \R$.  For simplicity, consider the representative case 
$\lambda_0 = 0$, and let $A_0 = I + \hat{T}(0)$, which is bounded pointwise by
$I + M(T)$ as an integral operator in $\B(X)$.    
One can write $\eta(\lambda/\epsilon)(I + \hat{T}(\lambda) - A_0)$ 
as the Fourier transform of
\begin{align*}
S_\epsilon(\rho) &= \epsilon\check{\eta}(\epsilon\,\cdot\,) * T(\rho) - 
\epsilon\check{\eta}(\epsilon\rho)(I-A_0)
\\
&= \int_\R \epsilon\big(\check{\eta}(\epsilon(\rho-\sigma))- 
\check{\eta}(\epsilon\rho)\big)
 T(\sigma)\,d\sigma.
\end{align*}
Here we are again suppressing the $(x,y)$ dependence to highlight the 
convolution in the $\rho$ variable.
The second equality uses the fact that $I - A_0(x,y) = -\hat{T}(0, x, y) = - 
\int_\R T(\rho,x,y)\,d\rho$.

By the mean value theorem, $\int_\R \epsilon |\check{\eta}(\epsilon(\rho-
\sigma))- \check{\eta}(\epsilon\rho)| \dd \rho \les \min(\epsilon|\sigma|, 1)$.  
As a consequence, for any $R > 0$
\begin{equation*}
M(S_\eps) \les \eps R M(T) + M(\chi_{|\rho| > R}T),
\end{equation*}
which immediately implies that $\norm[S_\eps][\U_X] \les \eps R \norm[T][\U_X] + 
\norm[\chi_{|\rho| > R}T][\U_X]$.
By Assumption C2 there is a choice of $R$ and $\eps$ to make this norm 
arbitrarily small.

For any smooth function $\phi$ supported in $[-
\frac{\epsilon}2,\frac{\epsilon}2]$, there exists a series
\begin{align*}
\phi(\lambda)(I + \hat{T}(\lambda))^{-1} 
&= \phi(\lambda)\big(A_0 + \eta(\lambda/\epsilon)(1 + \hat{T}(\lambda) - 
A_0)\big)^{-1}
\\
&= \phi(\lambda) A_0^{-1}\big(I + \hat{S_\epsilon}(\lambda)A_0^{-1}\big)^{-1}
 = \phi(\lambda) A_0^{-1} \sum_{k=0}^\infty (-1)^k 
      \big(\hat{S_\epsilon}(\lambda)A_0^{-1}\big)^k.
\end{align*}
When $\epsilon$ is chosen sufficiently small, the inverse Fourier transform of 
this series converges in $\U_X$.

For $\lambda$ in the compact interval $[-2L, 2L]$, there is a nonzero
lower bound on the length $\eps$ required for convergence of the power series.

Choose a finite covering of the compact set $[-2L, 2L]$ and a subordinated 
partition of the unity $(\phi_j)_j$ with
$\sum_j \phi_j = \eta(\lambda/2L)$, such that for each $j$ the local inverse 
$\phi_j\big((\lambda-\lambda_j)/\epsilon\big) \big((I + \widehat 
T(\lambda)\big)^{-1} \in \U_X$ is given
by an explicit series as above.  Thus $\eta(\lambda/2L) \big((I + \widehat 
T(\lambda)\big)^{-1}$ is the
sum of finitely many elements of $\U_X$.

\end{proof}

\begin{remark}
Condition C2 implies that $\hat{T}(\lambda)$ is a norm-continuous function
from $\R$ to $\B(X)$.  Condition C1 implies that $\lim_{|\lambda|\to \infty}
\norm[\hat{T}^N(\lambda)][\B(X)]= 0$.
\end{remark}

\section{Proof of the main result} \label{sec:proof}

Recall that $\hat{T}^\pm(\lambda) = VR_0^\pm(\lambda^2)$ as defined
in~\eqref{eq:T_hat}, hence 
$T^\pm(\rho,x,y)$ is represented by
the distribution kernel $V(x)/(4\pi|x-y|)$ supported on the surface $\{|x-y| 
\pm \rho = 0\}$.
The proof of Theorem \ref{main_theorem} is based on applying 
Proposition~\ref{thm:Wiener_2} to $T^\pm$ in the space $\U_{L^1(\R^3)}$.  
With the help of duality and algebraic relations,
this result will extend to $\U_\Kato$ and the family of interpolation spaces 
spanning them.

The pointwise invertibility of $I+\hat T^\pm(\lambda)$ in $\B(L^1)$ at each 
$\lambda$ follows by Fredholm's alternative from the absence of resonances 
or eigenvalues, once we show that $\hat T^\pm(\lambda)$ is a compact operator in 
$\B(L^1)$.

\begin{lemma}\lb{lemma_compact} $\hat T^\pm(\lambda)$ is a compact operator in 
$\B(L^1)$ for all $\lambda \in \R$.
\end{lemma}

\begin{proof}
By an approximation argument (since $V \in \Kato_0$), it suffices to treat the 
case when $V$ is smooth and compactly supported.  Then all functions 
$\hat{T}^\pm(\lambda)f = VR_0^\pm(\lambda^2)f$ are supported within ${\rm 
supp}\,V$.
In addition,
\begin{equation*}
(1 - \Delta)VR_0^\pm(\lambda^2)f = Vf + (1+\lambda^2)VR_0^\pm(\lambda^2)f -
2\nabla V \cdot \nabla R_0^\pm(\lambda^2)f - (\Delta V) R_0^\pm(\lambda^2)f.
\end{equation*}
Under the assumption that $V \in C^\infty_c(\R^3)$ each of the above terms belongs to 
$L^1(\R^3)$ with a
norm bound  of $C(1+\lambda^2)\norm[f][1]$.

For $\lambda$ fixed, $\hat{T}^\pm(\lambda)$ maps $L^1$ to $(1-\Delta)^{-1}L^1$ 
with fixed support inside
${\rm supp}\,V$.  Hence it is a compact operator on $L^1$ itself.
\end{proof}

Now the only obstacle to invertibility of $I + \hat{T}^\pm(\lambda)$ is the 
presence
of a nonzero solution to $\phi + VR_0^\pm(\lambda)\phi = 0$, $\phi \in L^1$.
At the same time, $R_0^\pm(\lambda)\phi$ would be a distributional solution to
$(-\Delta + V - \lambda)f = 0$.  We show below that such a scenario is not 
possible
under the spectral conditions of Theorem~\ref{main_theorem}.

\begin{lemma}\lb{comp_lemma} Assume that $V \in \Kato_0$, and suppose that
for some $\lambda \in [0,\infty)$ there is a nonzero $\phi \in L^1$
satisfying $\phi +V R_0^{\pm}(\lambda) \phi = 0$. 
Then in fact $\phi \in L^1 \cap \Kato$, and $R_0^\pm(\lambda)\phi \in 
\Kato^*\cap L^\infty$
belongs to $\la x\ra^\sigma L^2$ for each $\sigma > \frac12$.  

When $\lambda > 0$ the stronger conclusion $R_0^\pm(\lambda)\phi \in L^2$ is 
also valid.

Consequently, under the conditions of Theorem~\ref{main_theorem} 
the operators $I + R_0^{\pm}(\lambda) V$ and $I 
+ V R_0^{\pm}(\lambda)$ are invertible for all $\lambda \in [0, \infty)$ and
$$
(I + R_0^{\pm}(\lambda) V)^{-1} \in \B(L^{\infty}) \cap \B(\Kato^*),\ (I + V 
R_0^{\pm}(\lambda))^{-1} \in \B(L^1) \cap \B(\Kato).
$$
If $V \in L^{3/2,1}$, then $\Kato$ and $\Kato^*$ may be replaced respectively by
$L^{3/2,1}$ and $L^{3,\infty}$ in the above conclusion.
\end{lemma}

\begin{proof}
Choose an approximation $V_\eps \in C^b_c(\R^3)$ so that $\norm[V-V_\eps][\Kato] 
< 4\pi$.
Then $V_\eps R_0^\pm(\lambda)\phi$ belongs to both $L^1$ and $L^{3,\infty}$, 
which also
includes $\Kato$.  Thanks to the identity
\begin{equation*}
\phi + (V-V_\eps)R_0^\pm(\lambda)\phi = -V_\eps R_0^\pm(\lambda)\phi
\end{equation*}
and the smallness of $V-V_\eps$, we can rewrite
\begin{equation*}
\phi = -(I + (V-V_\eps)R_0^\pm(\lambda))^{-1}V_\eps R_0^\pm(\lambda)\phi
\end{equation*}
as an element of $L^1 \cap \Kato$.  Dominating the resolvent kernel by $(4\pi|x-
y|)^{-1}$
leads to the conclusion $R_0^\pm(\lambda)\phi \in \Kato^* \cap L^\infty \cap
\la x\ra^\sigma L^2$ for each $\sigma > \frac12$.

When $\lambda > 0$, one can use the fact that
\begin{equation*}
\Im[\la R_0^\pm(\lambda)\phi, \phi\ra] = -\Im[\la R_0^\pm(\lambda)\phi, 
VR_0^\pm(\lambda)\phi\ra] = 0
\end{equation*}
to conclude that the Fourier transform of $\phi$ vanishes on the sphere of 
radius $\sqrt{\lambda}$
in frequency space.  Then Corollary~4.2 of~\cite{GoSc04b} asserts that 
$R_0^\pm(\lambda)\phi \in L^2$.

The spectral assumptions in Theorem~\ref{main_theorem} rule out solutions of 
this kind.
Direct appliction of the Fredholm alternative then shows that $(I + 
VR_0^\pm(\lambda))^{-1}$
exists in $\B(L^1)$, and by duality $(I + R_0^\pm(\lambda)V)^{-1} \in 
\B(L^\infty)$.
The identity
\begin{equation*}
(I + VR_0^\pm(\lambda))^{-1} = I - V(I + VR_0^\pm(\lambda))^{-1}R_0^\pm(\lambda)
\end{equation*}
defines an inverse for $I + VR_0^\pm(\lambda)$ in $\B(\Kato)$ (also in 
$\B(L^{3/2,1})$ if 
$V \in L^{3/2,1}$) and the dual statement defines
an inverse for $I + R_0^\pm(\lambda)V$ in $\B(\Kato^*)$.
\end{proof}

The case $\lambda = 0$ generates several direct equivalences between
$-\Delta$ and $H$.

\begin{lemma} \label{lem:norm_equiv}
Let $V \in \Kato_0$, and suppose that $H = -\Delta + V$ does not have an
eigenvalue or resonance at zero.  Then $H\Delta^{-1}$ acts as an isomorphism
on both $L^1(\R^3)$ and $\Kato$ (and also on $L^{3/2,1}$ if $V \in L^{3/2,1}$).

Moreover, $H$ is an invertible linear map from $\dot{H}^1(\R^3)$ to its dual,
and the positive quadratic form $\la |H|^sf,f\ra$ is equivalent to 
$\norm[f][\dot{H}^s]^2 = \la (-\Delta)^s f,f\ra$ for $|s| \leq 1$.

Finally, assuming that $V \in L^{3/2, 1}$, then $\||H|^{s/2} f\|_{L^2} \sim 
\|f\|_{\dot H^s}$ whenever $|s| < 3/2$.
\end{lemma}

\begin{proof}
The statements about $H\Delta^{-1} = -(I + VR_0(0))$ acting on $L^1$
and on $\Kato$ are a restatement of Lemma~\ref{comp_lemma} with $\lambda = 0$.
It is well known that $V\in \Kato$ is form-bounded with respect to the
Laplacian (see for example \cite{Si82}, p.\ 459), hence $\Delta^{-1}H$ is a 
bounded map on $\dot{H}^1(\R^3)$.
Its invertibility (assuming zero is a regular point of the spectrum of $H$) 
follows from a similar compactness and Fredholm alternative argument.
We refer to~\cite{michael} for the details in the general case where $V$ is a
locally finite measure.

Note that $H - |H|$ is a finite linear combination of projections onto the
point spectrum of $H$, and the same is true of $H^{-1} - |H|^{-1}$
provided zero is not an eigenvalue or resonance.  Each eigenfunction
is exponentially decaying and belongs to
$\dot{H}^1(\R^3) \cap \dot{H}^{-1}(\R^3)$,
so the projections are bounded on any $H^s, |s| \leq 1$.
Then $|H|$ is also an isomorphism between $\dot{H}^1(\R^3)$ and its dual space.
Its Hermitian square root $\sqrt{|H|}:\dot{H}^1 \to L^2$ is another
isomorphism, 
meaning $\la |H|f,f\ra = \norm[\sqrt{|H|}f][2]^2 \sim~\norm[f][\dot{H}^1]^2$.

Since the quadratic forms $\langle |H| f, f \rangle$ and $\langle -\Delta f, f 
\rangle$ are equivalent --- each is form-bounded by a multiple of the other --- 
the same holds for $\langle |H|^s f, f \rangle$ and $\langle (-\Delta)^s f, f 
\rangle$ for $0 \leq s \leq 1$. Thus $\||H|^{s/2} f\|_{L^2} \sim \|f\|_{\dot 
H^s}$ for $0 \leq s \leq 1$ and then by duality the same is true when $-1 \leq s 
\leq 0$.

Next, assume that $V \in L^{3/2, 1}$. Since the eigenstates of $H$ are in $\dot 
H^1$, by bootstrapping in the eigenstate equation $f = R_0(E) V f$, $E<0$, we 
first obtain that $f \in L^\infty$, then that $(-\Delta +1) f \in L^{3/2, 1}$. 
Consequently $f \in \dot H^s$ for any $s \in (-3/2, 3/2)$ --- due to exponential 
decay for the negative range.

Further note that now $H=-\Delta+V$ is a bounded operator from $\dot H^s$ to 
$\dot H^{s-2}$ for any $s \in (1/2, 3/2)$, hence $(-\Delta)^{-1} H$ is a bounded 
map on $\dot H^s$. Its invertibility is a consequence of Lemma \ref{comp_lemma}. 
Thus $|H|$ is an isomorphism between $\dot H^s$ and $\dot H^{s-2}$. Since $\dot 
H^s$ and $\dot H^{s-2}$ are dual with respect to the $\dot H^{s-1}$ dot product, 
we obtain that $\sqrt {|H|} \in \B(\dot H^s, \dot H^{s-1})$ is an isomorphism. 
Since $|H|^{(s-1)/2}$ is an isomorphism from $\dot H^{s-1}$ to $L^2$, we obtain 
that $|H|^{s/2}$ is an isomorphism from $\dot H^s$ to $L^2$ for any $s \in [0, 
3/2)$. The conclusion extends by duality to $s \in (-3/2, 0]$.
\end{proof}

It is now an exercise to show that $T^\pm(\rho,x,y)$ falls within the
framework of Proposition~\ref{thm:Wiener_2}.  Without loss of
generality, the result is stated in terms of $T^-$ alone.

\begin{theorem} \label{thm:goal_lattice}
Let $V \in \Kato_0$ be a scalar potential in $\R^3$ satisfying the assumptions
of Theorem~\ref{main_theorem}.
Then
\begin{align*}
\norm[T^-(\rho)][\U_X] \le 
  \frac{\norm[V][\Kato]}{4\pi}  \\
\text{and}\quad 
 \bignorm[(\1 + T^-)^{-1}][\U_X] < \infty,
\end{align*}
where $X$ may be any one of the function spaces
$\Kato^\theta, 0 \leq \theta \leq 1$.  In particular this includes
$\Kato = \Kato^1$ and $L^1(\R^3) = \Kato^0$.

If $V \in L^{3/2,1}$ then one may also choose $X$ to be $L^{3/2,1}$
or any $L^{p,q}$, $1 < p < \frac32$, $1 \leq q \leq \infty$ by
replacing the right side of the first inequality with $C \norm[V][3/2,1]$.
\end{theorem}

Before attempting a proof that spans the gamut of admissible spaces
$X$, it will be convenient to introduce one additional bit of abstract
notation.  Let $X$ and $Y$ be two Banach lattices of functions
over $(\X, \mu)$ and $({\mathcal Y}, \nu)$ respectively.

\begin{definition}
Let $\U_{X, Y}$ be the set of bounded operators from $Y_y \M_{\rho}$ to
$X_x \M_{\rho}$ of the form
$$
(T f)(\rho,x) = \int_{\mathcal Y} \int_{-\infty}^{\infty}
f(\rho-\sigma, y) T(\sigma, x, y) \dd \sigma \,\nu(dy),
$$
where $T(\rho, x, y)$ has the property that $T(\rho, x, y) \dd \rho \in 
\M_{\rho}$ for a.e.\ $x$ and $y$ and
$$
M(T)(x,y) = \int_{-\infty}^{\infty} |T(\rho, x, y)| \dd \rho 
$$
is the integral kernel for a bounded operator from $Y$ to $X$.
$\U_{X, Y}$ is a Banach space under the norm $\|T\|_{\U_{X, Y}} = 
\|M(T)\|_{\B(Y, X)}$
\end{definition}

Though in general $\U_{X, Y}$ cannot be an algebra under the composition of 
operators, it has the more general property 
that for any three Banach lattices $X$, $Y$, and $Z$
$$
\|S T\|_{\U_{X, Z}} \leq \|S\|_{\U_{X, Y}} \|T\|_{\U_{Y, Z}}.
$$
This structure is an algebroid.

The Fourier transform defined by (\ref{fourier}) fulfills $\|\hat 
T(\lambda)\|_{\B(Y, X)} \leq \|T\|_{\U_{X, Y}}$ and $(ST)^{\wedge}(\lambda) = 
\hat S(\lambda) \hat T(\lambda)$. Kernels $T \in \U_{X, Y}$ define, through 
(\ref{image}), bounded operators from $Y$ to $X_x \M_{\rho}$.

\begin{proof}[Proof of Theorem \ref{thm:goal_lattice}]

In our previously introduced notation, $T^-$ has the kernel
$$
T^-(\rho, x, y) = (4\pi \rho)^{-1} V(x) \delta_{|x-y|}(\rho).
$$
This is a finite measure in $\rho$, for all $x$ and $y$, and
$$
M(T^-)(x,y) = \int_{-\infty}^{\infty} |T^-(\rho, x, y)| \dd \rho 
= \frac {|V(x)|}{4\pi|x-y|}.
$$
$M(T^-)$ belongs to $\B(L^1)$ exactly when $V \in \Kato$. Indeed,
\begin{equation*}\begin{aligned}
\Big\|\int_{\R^3} \frac {|V(x)|}{4\pi|x-y|} f(y) \dd y\Big\|_{L^1_x} &= 
\int_{\R^3} \int_{\R^3} \frac {|V(x)|}{4\pi|x-y|} |f(y)| \dd y \dd x \\
&\leq \sup_y \int_{\R^3} \frac {|V(x)|}{4\pi|x-y|} \dd x \int_{\R^3} |f(y)| \dd 
y \leq \frac{\|V\|_{\Kato}}{4\pi} \|f\|_1
\end{aligned}\end{equation*}
with equality being nearly achieved if $f$ is concentrated on points $y \in 
\R^3$
that nearly optimize the supremum in the Kato norm.
These computations show that $T^- \in \U_{L^1}$ and $\|T^-\|_{\U_{L^1}} \le 
{\|V\|_{\Kato}}/{4\pi}$.

To work in other spaces, consider $\check{R}^-(\rho,x,y) = (4\pi\rho)^{-1}
\delta_{|x-y|}(\rho)$.  Its Fourier transform is the family of free resolvents
$R_0^-(\lambda^2)$.  An elementary calculation similar to the one above
shows that $\check{R}^- \in \U_{\Kato^*, L^1} \cap \U_{L^\infty, \Kato}$ with 
norm $1/4\pi$.
Multiplication by $V$ maps $L^\infty$ to $\Kato$, and maps $\Kato^*$
to $L^1$, so $T^- = V\check{R}^-$ belongs to $\U_\Kato$.

Suppose the second conclusion of the theorem is valid for $X = L^1(\R^3)$.
Then the resolvent identity $(I + \hat{T}^-(\lambda))^{-1}
= I - V(I + \hat{T}^+(\lambda)^*)^{-1}R_0^-(\lambda)$
leads to a bound
\begin{equation*}
\norm[(\1 + T^-)^{-1}][\U_\Kato]
\le 1 + \norm[V][\Kato]\bignorm[\big((\1 + T^+)^*\big)^{-1}][\U_{L^\infty}]
\norm[\check{R}^-][\U_{L^\infty},\Kato] < \infty
\end{equation*}
The remaining cases $X = \Kato^\theta$ follow by interpolation.
If $V \in L^{3/2,1}$ the argument can be repeated for Lorentz spaces by
embedding $L^{3/2,1} \subset \Kato$ and $\Kato^* \subset L^{3,\infty}$.

The case $X = L^1(\R^3)$ was treated in~\cite[Theorem 2]{becgol}
as an application of Proposition~\ref{thm:Wiener_2}, and is summarized below.

We have determined that $T^- \in \U_{L^1}$, with $\norm[T^-][\U_{L^1}]
\leq \norm[V][\Kato]/4\pi$.  Lemma~\ref{comp_lemma} shows that
$I + \hat{T}^-(\lambda)$ is an invertible element of $\B(L^1)$ for each
$\lambda \in \R$ provided that hypotheses of Theorem~\ref{main_theorem}
are satisfied.

When applied to $T^-$ specifically, condition (C2) of Proposition 
\ref{thm:Wiener_2}
is equivalent to the statement
$$
\lim_{R \to \infty} \Big\|\chi_{|x-y|>R} \frac {|V(x)|}{4\pi|x-
y|}\Big\|_{\B(L^1)} = 0,
$$
which in turn reduces to the distal Kato property (\ref{eq:distalKato}).

It suffices to verify condition (C1) for bounded and compactly supported
potentials.  Norm-continuity of the mapping
$V \in \Kato \mapsto T^- \in \U_{L^1}$ allows the extension of
property (C1) to all potentials $V \in \Kato_0$.

If $V \in C^b_c(\R^3)$, then $\hat{T}^-(\lambda) = VR_0^-(\lambda^2)$
maps $L^1(\R^3)$ to $L^{4/3}(\R^3)$ and vice versa.  
By Theorem 2.3 of~\cite{KeRuSo} and scaling, it is also true that
\begin{equation*}
\norm[\hat{T}^-(\lambda)f][\frac{4}{3}] \les
|\lambda|^{-1/2}\norm[V][2]\norm[f][\frac{4}{3}].
\end{equation*}
Then $\norm[\hat{T}^-(\lambda)^{10}][\B(L^1)]
\les (1 + \lambda^2)^{-2}$, which is more than sufficient to imply that
$\norm[\partial_\rho^k (T^-)^{10}(\rho,x,y)][\B(L^1)]$ is uniformly bounded
over all $\rho \in \R$ and $k = 0,1$.

The compact-approximation property (C2) is preserved by products in $\U_{L^1}$,
therefore it suffices to verify that $\eta(\rho)(T^-)^{10}$ satisfies condition 
(C1) for all
compactly supported functions $\eta$.  However
$\norm[\partial_\rho\, \eta (T^-)^{10}(\rho,x,y)][\B(L^1)]$ will be uniformly 
bounded,
and vanishes away from the support of $\eta$.  It follows that
\begin{align*}
&\bignorm[\eta(T^-)^{10}(\rho,x,y) - \eta(T^-)^{10}(\rho-\delta,x,y)][\U_{L^1}] 
\\
&\leq \  \int_\R \bignorm[|\eta (T^-)^{10}(\rho,x,y)
- \eta (T^-)^{10}(\rho-\delta,x,y)|][\B(L^1)]\,d\rho
\leq C \delta
\end{align*}
which converges to zero as $\delta \to 0$.
The constant depends on parameters such as 
the size and support of $V$ and of $\eta$ but is independent of $\delta$.
\end{proof}

A kernel's membership in $\U_X$ ensures only integrability in the $\rho$ 
variable. Several statements within Theorem \ref{main_theorem} require
an additional weighted integrability condition $\rho T(\rho,x,y) \in \U_{\Kato, 
L^1}$.
In this setting, we have the  following extension of Theorem 
\ref{thm:goal_lattice}:
\begin{proposition} \label{thm:2ndgoal}
Assume that $V \in \Kato_0$ satisfies the conditions of 
Theorem~\ref{main_theorem}.
Then
\begin{align*}
\norm[\rho T^-][\U_{\Kato, L^1}] \le 
  \frac{\norm[V][\Kato]}{4\pi} \\
\text{and}\quad 
 \bignorm[\rho (\1 + T^-)^{-1}
  ][\U_{\Kato, L^1}] < \infty.
\end{align*}
\end{proposition}
\begin{proof}
$\rho T^-(\rho)$ has the kernel
$$
\rho T^-(\rho)(x, y, \rho) = (4\pi)^{-1} V(x) \delta_{|x-y|}(\rho).
$$
Therefore
$$
M(\rho T^-(\rho))(x, y) = (4\pi)^{-1} |V(x)|
$$
is the kernel of a bounded operator from $L^1$ to $\Kato$ as long as $V \in 
\Kato$. 
This implies that $\norm[\rho T^-][\U_{L^1, \Kato}] \le 
\frac{\norm[V][\Kato]}{4\pi}$.

The second inequality is obtained from Theorem \ref{thm:goal_lattice}, because 
of the fact that
\begin{equation}\lb{derivata}
\rho\big(\mathbf 1 +T)^{-1} = -(\mathbf 1 + T)^{-
1}(\rho)\big(\rho T(\rho)\big)(\mathbf 1 + T)^{-1}(\rho).
\end{equation}
By taking the Fourier transform, (\ref{derivata}) is equivalent to
\begin{equation}\lb{rezol}
\partial_{\lambda}(I + \hat T(\lambda))^{-1} = - (I + \hat 
T(\lambda))^{-1} \partial_{\lambda} \hat T(\lambda) (I + \hat T(\lambda))^{-1}.
\end{equation}
To prove (\ref{rezol}), one divides the formula
\begin{equation*}
(I + \hat T(\lambda_1))^{-1} - (I + \hat T(\lambda_2))^{-1} = (I + \hat 
T(\lambda_1))^{-1} (\hat T(\lambda_2) - \hat T(\lambda_1)) (I + \hat 
T(\lambda_2))^{-1}
\end{equation*}
by $\lambda_1-\lambda_2$ and passes to the strong limit.
\end{proof}

It is also useful to obtain bounds on $\rho^\theta T(\rho, x, y)$ for $0 \leq 
\theta \leq 1$:
\begin{corollary}\lb{general}
Assume that $V \in \Kato_0$ satisfies the conditions of 
Theorem~\ref{main_theorem}.
Then, for $0 \leq \theta_1 \leq \theta_2 \leq 1$,
\begin{align*}
\norm[\rho^{\theta_2-\theta_1} T^-][\U_{\Kato_{\theta_2}, \Kato_{\theta_1}}] \le 
  \frac{\norm[V][\Kato]}{4\pi} \\
\text{and}\quad 
 \bignorm[|\rho|^{\theta_2-\theta_1} (\1 + T^-)^{-1}
  ][\U_{\Kato_{\theta_2}, \Kato_{\theta_1}}] < \infty.
\end{align*}
\end{corollary}
In fact, all these measures have positive support in $\rho$, so the absolute 
value is superfluous.
\begin{proof} The statements follow by complex interpolation between the three 
cases addressed in Theorem \ref{thm:goal_lattice} and Proposition 
\ref{thm:2ndgoal}.
\end{proof}

Finally, we can prove the main result of this paper.
\begin{proof}[Proof of Theorem \ref{main_theorem}]

Start with the functional calculus formula in~\eqref{eq:Stonesine},
\begin{align*}
\frac{\sin(t\sqrt{H})P_c}{\sqrt{H}}f &= \frac{1}{\pi i}\int_{-\infty}^\infty 
\sin(t\lambda) R_V^+(\lambda^2) f\,d\lambda \\
&= -\frac{1}{2\pi}\int_{-\infty}^\infty (e^{it\lambda}-e^{-it\lambda})
R_V^+(\lambda^2) f\,d\lambda \\ 
&= -\frac{1}{2\pi} \int_{-\infty}^\infty (e^{it\lambda}-e^{-it\lambda})
(I +\hat{T}^-(\lambda)^*)^{-1}R_0^+(\lambda^2) f\,d\lambda
\end{align*}
This is a symmetrization of the (inverse) Fourier transform of
$R_V^+(\lambda^2)$.  Most of the desired estimates do not rely on cancellation
so it suffices to bound the Fourier transform in $L^p_t$ by itself. 

In the event that $V = 0$, the factor $(I + \hat{T}^-(\lambda)^*)^{-1}$
reduces to the identity operator and one is left to calculate the inverse 
Fourier
transform of $R_0^+(\lambda^2)$.  Thanks to the explicit
free resolvent kernel $R_0^+(\lambda^2)(x,y) = 
(4\pi|x-y|)^{-1}e^{i\lambda|x-y|}$, 
the result is a measure supported on the light cone $|x-y| = |t|$.
In the notation of Section~\ref{sec:Wiener},
\begin{equation*}
\check{R}_0^+(\rho,x,y) = (4\pi \rho)^{-1}\delta_0(\rho+|x-y|),
\end{equation*}
which is precisely the backward propagator of the free wave equation.
The symmetrized version contains both the forward and backward propagators.

Direct inspection of of the operator kernel shows that
\begin{align*}
\norm[\rho \check{R}_0^+][\U_{L^\infty, L^1}] &= \frac{1}{2\pi}\\ 
\norm[\check{R}_0^+][\U_{L^\infty, \Kato}] &= \frac{1}{2\pi},\
\norm[\check{R}_0^+][\U_{\Kato^*,L^1}] = \frac{1}{2\pi}, \\
\end{align*}
where we have taken $\check{R}_0^+$ to indicate the inverse Fourier transform
of $R_0^+(\lambda^2)$ in a convenient abuse of notation.

Let $S := (\1 + T^-)^{-1}$ be defined in $\U_X$ (for various 
function spaces $X$) according to Theorem~\ref{thm:goal_lattice}.  Then
$S^* \in \U_{X^*}$ satisfies 
$(S^*)^\wedge (\lambda) = (I + \hat T^-(\lambda)^*)^{-1}$.  Observe that
$\rho \int_\R e^{i\rho\lambda}\hat{S}^*(\lambda)R_0^+(\lambda^2)\,d\lambda =
(\rho S^*)\check{R}_0^+ + S^*(\rho\check{R}_0^+)$,
where the products on the right-hand side are taken in the algebroid
structure of $\U_{X,Y}$.
\begin{align*}
\norm[\rho \check{R}_V^+][\U_{L^\infty, L^1}] &\leq
\norm[\rho S^*][\U_{L^\infty, \Kato^*}] \norm[\check{R}_0^+][\U_{\Kato^*, L^1}]
\\
&+ \norm[S^*][\U_{L^\infty}] \norm[\rho \check{R}_0^+][\U_{L^\infty, L^1}]
< \infty
\end{align*}
by the collected results in Theorem~\ref{thm:goal_lattice} and
Proposition~\ref{thm:2ndgoal}.  It immediately follows that 
$\bignorm[t \frac{\sin(t\sqrt{H})}{\sqrt{H}}P_c f][L^\infty_xL^1_t] \les
\norm[f][1]$.  The subsequent two inequalities in Theorem~\ref{main_theorem}
express the fact that the tail integral of a function in $t^{-1}L^1_t$ belongs
to both $L^1_t$ and $t^{-1}L^\infty_t$.
After substituting $Hf$ in place of $f$, one obtains that
\begin{equation*}
\Bignorm[\int_t^\infty \sqrt{H} \sin(s\sqrt{H})P_c f\,ds][L^\infty_x]
\les t^{-1}\norm[Hf][1]
\end{equation*}
for all $t >0$.  Typically 
$\norm[Hf][1] = \norm[(I + V(-\Delta)^{-1})\Delta f][1]
\les \norm[\Delta f][1]$.  Lemma~\ref{lem:norm_equiv} asserts that
the two norms are even equivalent if $H$ has no resonance or eigenvalue
at zero.

The same tail integral bounds also hold for $\cos(t\sqrt{H})P_c f$
once it is established that $\cos(t\sqrt{H})P_c =
\int_t^\infty \sqrt{H}\sin(s\sqrt{H})P_c\,ds$ in an appropriate sense.
Note that the difference
\begin{equation*}
Af := \cos (t\sqrt{H})P_c f - \int_t^\infty \sqrt{H}\sin(s\sqrt{H})P_c f\,ds
\end{equation*}
is independent of $t$ and is a bounded linear operator from
$L^2 \cap H^{-1}L^1$ to $L^2 + L^\infty$.  By Lemma~\ref{lem:norm_equiv}
it is permissible to replace $H^{-1}L^1$ with the equivalent space
$(-\Delta)^{-1}L^1$.  However $\cos(t\sqrt{H})P_c$ converges weakly to zero
in $\B(L^2)$ as $t \to \infty$, and the norm of
$\int_t^\infty \frac{\sin(s\sqrt{H})}{\sqrt{H}}P_c\,ds$
in $\B(L^1, L^\infty)$ is dominated by $t^{-1}$.
That forces $\la Af, g\ra = 0$ for any pair of test functions $f,g$, which means
$A=0$.

The immediate consequence is that $\norm[t \cos(t\sqrt{H})P_c f][L^\infty_x]$
and $\norm[\cos(t\sqrt{H})P_c f][L^\infty_xL^1_t]$ are both controlled by
$\norm[\Delta f][1]$, which is comparable to $\norm[Hf][1]$.  A second 
integration
in the $t$ direction shows that
$\bignorm[\frac{\sin(t\sqrt{H})}{\sqrt{H}}P_c f][L^\infty_{x,t}] \les
\norm[\Delta f][1]$.

Estimates for $\frac{\sin(t\sqrt{H})}{\sqrt{H}}P_c f$ in unweighted $L^1_t$ are
straightforward, requiring only the composition
\begin{align*}
\norm[\check{R}_V^+][\U_{L^\infty, \Kato}] &\leq
\norm[S^*][\U_{L^\infty}] \norm[\check{R}_0^+][\U_{L^\infty, \Kato}] \\
\norm[\check{R}_V^+][\U_{\Kato^*, L^1}] &\leq
\norm[S^*][\U_{\Kato^*}] \norm[\check{R}_0^+][\U_{\Kato^*, L^1}].
\end{align*}

It follows that
$\bignorm[\frac{\sin(t\sqrt{H})}{\sqrt{H}}P_c][L^\infty_xL^1_t]
\les \norm[f][\Kato]$ and 
$\bignorm[\frac{\sin(t\sqrt{H})}{\sqrt{H}}P_c][\Kato^*_xL^1_t]
\les \norm[f][1]$.  Integrating in from $t = \infty$ as above
proves the corresponding estimates for $\cos(t\sqrt{H})P_c$ in
$L^\infty_t$, with norm controlled by $\norm[\Delta f][\Kato]$ and
$\norm[\Delta f][1]$ respectively.  Lemma~\ref{lem:norm_equiv} is
used to show the equivalence of $\Delta f$ and $H f$ in the Kato norm.

Initial conditions $f(x)$ belonging to a Sobolev space such as
$\dot{W}^{1,1}(\R^3)$ require extra care because the Banach lattice
structure is absent.  Indeed, $I + \hat{T}^-(\lambda)$ need not be a
bounded operator here, nor possess a bounded inverse.

We consider the action of $\cos(t\sqrt{H})P_c$ on functions with one
weak derivative.  Start with the spectral representation~\eqref{eq:Stonecosine}
to derive
\begin{align*}
\cos(t\sqrt{H})P_c f &= \frac{1}{\pi i}\int_{-\infty}^\infty 
\lambda \cos(t\lambda) R_V^+(\lambda^2) f\,d\lambda \\
&= \frac{1}{2\pi i}\int_{-\infty}^\infty (e^{it\lambda}+e^{-it\lambda})
\lambda R_V^+(\lambda^2) f\,d\lambda \\
&= \frac{1}{2\pi i} \int_{-\infty}^\infty (e^{it\lambda}+e^{-it\lambda})
(I +\hat{T}^-(\lambda)^*)^{-1}\lambda R_0^+(\lambda^2) f\,d\lambda
\end{align*}
As before it suffices to bound the inverse Fourier transform of
$\lambda R_V^+(\lambda^2) f$ alone rather than the symmetrized version.

In the free case $(\lambda R_0^+(\lambda^2) f)^\vee(\rho)$ provides a solution
to the wave equation in $\R^-_\rho \times \R^3_x$ with boundary conditions
$u(x,0) = f$, $u_\rho(x,0) = 0$.  The explicit formula is
\begin{equation*}
\frac{1}{2\pi i} (\lambda R_0^+(\lambda^2) f)^\vee(\rho,x) 
= \int_{S^2} f(x+ \rho\omega) + \rho \partial_\rho 
f(x+\rho\omega) \dd \omega
\end{equation*}
This is an even function of $\rho$, so for the purposes of estimating its norm 
in
$L^1_\rho$ we may integrate over $\R^+$ instead of $\R^-$.  If $f$ is
bounded and has compact support then
\begin{align*}
\bignorm[(\lambda R_0^+(\lambda^2) f)^\vee(\,\cdot\,,x)][L^1_\rho] 
&\leq C \int_0^\infty \int_{S^2} |f(x+ \rho\omega)| + 
|\rho\, \partial_\rho f(x+\rho\omega)| \dd \omega d\rho \\
&\les (|f| * |x|^{-2})(x) + (|\nabla f| * |x|^{-1})(x) \\
&\les (|\nabla f| * |x|^{-1})(x),
\end{align*}
because $|f(x)| \les |\nabla f| * |x|^{-2}$ pointwise almost everywhere.
Consequently
\begin{equation*}
\bignorm[(\lambda R_0^+(\lambda^2) f)^\vee][\Kato^*_xL^1_\rho] \les
\norm[\nabla f][1] \quad {\rm and} \quad
\bignorm[(\lambda R_0^+(\lambda^2) f)^\vee][L^\infty_xL^1_\rho] \les
\norm[\nabla f][\Kato].
\end{equation*}
The fact that $S^* \in \U_{\Kato^*} \cap\, \U_{L^\infty}$, together with 
the prior estimates for $(\lambda R_0^+(\lambda^2)f)^\vee$, leads
to immediate bounds of the form
\begin{align*}
\norm[\cos(t\sqrt{H})P_c f][\Kato^*_xL^1_t] &\les \norm[\nabla f][1] \\
\norm[\cos(t\sqrt{H})P_c f][L^\infty_x L^1_t] &\les \norm[\nabla f][\Kato].
\end{align*}
Integrating this pair of inequalities with respect to $t$ 
produces the further bounds
$\bignorm[\frac{\sin(t\sqrt{H})}{\sqrt{H}}P_c f][\Kato^*_xL^\infty_t]
\les \norm[\nabla f][1]$ and
$\bignorm[\frac{\sin(t\sqrt{H})}{\sqrt{H}}P_c f][L^\infty_{x,t}]
\les \norm[\nabla f][\Kato]$.

Time-weighted bounds in the free case are also derived from the explicit form
of the propagator kernel.
\begin{align*}
\bignorm[\rho (\lambda R_0^+(\lambda^2) f)^\vee(\,\cdot\,,x)][L^1_\rho] 
&\leq C \int_0^\infty \int_{S^2} |\rho f(x+ \rho\omega)| + 
|\rho^2\, \partial_\rho f(x+\rho\omega)| \dd \omega d\rho \\
&\les |f| * |x|^{-1} + \norm[\nabla f][1] \\
&\les \norm[\nabla f][1].
\end{align*}
which means that
$\bignorm[\rho (\lambda R_0^+(\lambda^2) f)^\vee][L^\infty_xL^1_\rho]
 \les \norm[\nabla f][1]$.  Lemma~\ref{kato_w11} was used in the last line to
control the size of $|f| * |x|^{-1}$.

Proposition~\ref{thm:2ndgoal} provides the additional information that
$\rho S^* \in \U_{L^\infty, \Kato^*}$.  Imitating the weighted
$L^1$ arguments for the sine propagator,
\begin{align*}
\norm[\rho(\lambda R_V^+)^\vee f][L^\infty_x L^1_t]
&\leq  \norm[\rho S^*][\U_{L^\infty, \Kato^*}] 
\norm[(\lambda R_0^+)^\vee f][\Kato^*_x L^1_\rho] 
+ \norm[S^*][\U_{L^\infty}] \norm[\rho(\lambda R_0^+)^\vee f][L^\infty_x 
L^1_\rho] \\
&\les \norm[\nabla f][1]
\end{align*}
The same bound is true of $\rho(\lambda R_V^-)^\vee f$, and their sum 
shows that 
\begin{equation*}
\norm[t \cos(t\sqrt{H})P_c f][L^\infty_x L^1_t] \les \norm[\nabla f][1].
\end{equation*}

Once again the tail integral of a function in $t^{-1}L^1_t$ belongs both to
$L^1_t$ as well as $t^{-1}L^\infty_t$.  This provides mapping bounds for
$\frac{\sin(t\sqrt{H})}{\sqrt{H}}P_c$ once it is established that the operator
\begin{equation*}
Bf := \frac{\sin(t\sqrt{H})}{\sqrt{H}}P_c f + \int_t^\infty \cos(t\sqrt{H})P_c 
f\,ds
\end{equation*}
is trivial.  Here $B$ is bounded from $L^2 \cap \dot{W}^{1,1}$  to
$\dot{H}^1 + L^\infty$ (recalling that $\dot{H}^1 \cong H^{-1/2} L^2$),
and for any pair of smooth test functions, $f, g$
one can show that $\lim_{t\to\infty} \la Bf, g\ra = 0$.  Since $Bf$ is in fact
independent of $t$ it follows that $B = 0$.

The first consequence, that
$\bignorm[\frac{\sin(t\sqrt{H})}{\sqrt{H}}P_c f][L^\infty_xL^1_t] \les
\norm[\nabla f][1]$, was proved earlier in the discussion with a stronger bound
in terms of $\norm[f][\Kato]$ in place of $\norm[\nabla f][1]$.  The second
consequence is the dispersive estimate
\begin{equation}
\Bignorm[\frac{\sin(t\sqrt{H})}{\sqrt{H}}P_c f][L^\infty]
\les |t|^{-1}\norm[\nabla f][1].
\end{equation}

The inhomogeneous estimates presented in Theorem~\ref{main_theorem}
are elementary extensions of the dispersive bounds proved above.
They reduce to the following statements about the propagator kernels
for $\frac{\sin(t\sqrt{H})}{\sqrt{H}}P_c$ and $\frac{\cos(t\sqrt{H})}{H}P_c$,
denoted by $K_1(t,x,y)$ and $K_2(t,x,y)$ respectively:
\begin{equation} \label{eq:prop_kernels}
\begin{aligned}
\norm[K_1(\,\cdot\,,x,y)][\meas] \ &{\rm defines\ bounded\ maps\ taking\ }
L^1\to \Kato^* \ {\rm and}\ \Kato\to L^\infty. \\
\norm[K_2(\,\cdot\,,x,y)][\infty] \ &{\rm defines\ bounded\ maps\ taking\ }
L^1\to \Kato^* \ {\rm and}\ \Kato\to L^\infty. \\
\norm[K_2(\,\cdot\,,x,y)][1] \ &\les 1 \ {\rm at\ almost\ every\ }x\ 
{\rm and}\ y.
\end{aligned}
\end{equation}
Each of these is a well known fact when $H = -\Delta$, and is readily
transferred to $H = -\Delta + V$ by composition with the operator
$S^* \in \U_{\Kato^*} \cap \U_{L^\infty}$.
\end{proof}

\begin{proof}[Proof of Corollary \ref{cor_interpolation}]
The inhomogeneous propagator estimates involving $\Kato^\theta$ and its dual
all follow from complex interpolation of the kernel bounds
in~\eqref{eq:prop_kernels}, recalling that $\Kato^0 = L^1(\R^3)$.
The fact that $K_2$ belongs to $L^1_t \cap L^\infty_t$ allows for a second
parameter of interpolation.

Each of the homogeneous estimates that follow is an interpolation between
three bounds stated in Theorem~\ref{main_theorem}.  For example one is given
that $\frac{\sin(t\sqrt{H})}{\sqrt{H}}P_c f$ is controlled in both the
$\Kato^*_x L^\infty_t$ and the $L^\infty_x L^1_t$ norms by
$\norm[\nabla f][1]$, and controlled in the $L^\infty_{x,t}$ norm by
$\norm[\nabla f][\Kato]$.  The time-decay estimates as stated in 
Corollary~\ref{cor_interpolation} are in fact slightly weaker than what
naturally arises via this method.  A more precise statement is
\begin{align*}
\Bignorm[t^{1-\theta_1-\theta_2} \frac{\sin(t\sqrt{H})P_c}{\sqrt{H}} f][
(\Kato^{\theta_2})^*_x L^\infty_t] &\les \norm[\nabla f][\Kato^{\theta_1}] \\
\bignorm[t^{1-\theta_1 - \theta_2} \cos(t\sqrt{H})P_c f][
(\Kato^{\theta_2})^*_x L^\infty_t] &\les \norm[\Delta f][\Kato^{\theta_1}]
\end{align*}

The first two Lorentz-space inequalities are also derived
from~\eqref{eq:prop_kernels}, this time using real interpolation.
Recall that $L^{3/2,1} \subset \Kato$ and $\Kato^* \subset L^{3,\infty}$.
Then $\norm[K_1(\,\cdot\, ,x, y)][\meas]$ defines a linear map from
$L^1$ to $L^{3,\infty}$, and $L^{3/2,1}$ to $L^\infty$.  By Hunt's
interpolation theorem~\cite{hunt} the same map is also bounded from 
$L^{p,s}$ to $L^{q,s}$ so long as $\frac{1}{p} - \frac{1}{q} = \frac23$
and $1 < p,q < \infty$.  For the cosine propagator $K_2$, first use complex
interpolation as above to show that $\norm[K_2(\,\cdot\, ,x,y)][p]$
maps $L^1$ to $(\Kato^{1-1/p})^* \subset L^{3p/(p-1),\infty}$
and $L^{3p/(2p+1),1} \subset \Kato^{1-1/p}$ to $L^\infty$.
Hunt's theorem establishes mapping bounds in the intermediate Lorentz spaces
and the convolution with $L^p_t$ is handled by Young's inequality.

The dispersive estimate (\ref{eq:dispLorentz}) with polynomial time-decay is 
derived
from~\eqref{eq:dispKato}. By Lemma~\ref{lem:norm_equiv} and complex 
interpolation,
$\Delta f$ and $Hf$ have equivalent norms in $\Kato^\theta$ for all
$0 \leq \theta \leq 1$.  Once $H$ has been brought back to the left side of the
inequality, apply real interpolation between the endpoint cases
$(\theta_1, \theta_2) = (0, \frac{r-1}{r})$  and
$(\theta_1, \theta_2) = (\frac{r-1}{r}, 0)$.

The final two bounds are derived in the same manner.  An additional
assumption that $V \in L^{3/2,1}$ is needed to ensure that $\Delta f$ and $H f$
are comparable in the function space $L^{3/2,1}$ as well as in $L^1$.
The norm equivalence holds in all intermediate Lorentz spaces
$L^{p,s}$, $1 < p < \frac32$ as well.
\end{proof}

\begin{proof}[Proof of Theorem \ref{thm_maximal}] 

Split the operator $\frac{e^{it\sqrt{H}}}{\sqrt{H}}P_c = 
\frac{\cos(t\sqrt{H})}{\sqrt{H}}P_c + i\frac{\sin(t\sqrt{H})}{\sqrt{H}}P_c$
into its real and imaginary components.  Since the continuous spectrum of
$H$ lies on $[0,\infty)$, both parts are well defined and (up to the factor of
$i$) self-adjoint.  Both the regular and reversed-norm Strichartz inequalities
are proved by applying $TT^*$ argument to each component separately.
We begin with the reversed-norm
estimates~\eqref{eq:StrichKato}--\eqref{eq:StrichSobolev}, first in the energy-
critical case.

Consider the operator $\displaystyle T f 
= \frac {\sin(t\sqrt H) P_c}{\sqrt H} f$, with $t$ ranging over the entire real 
line.
Then $\displaystyle T^* F = \int_{\R} 
\frac {\sin(t\sqrt H) P_c}{\sqrt H} F(t) \dd t$. Consequently
$$\begin{aligned}
(T T^* F)(t) &= \int_{\R} \frac {\sin(t\sqrt H) P_c}{\sqrt H} \frac {\sin(s\sqrt 
H) P_c}{\sqrt H} F(s)\dd s \\
&=  \frac 1 2 \int_{\R} \Big(\frac {\cos((t-s)\sqrt H) P_c}{H} - \frac 
{\cos((t+s)\sqrt H) P_c}{H}\Big) F(s)\dd s.
\end{aligned}$$
By the combined estimates~\eqref{eq:cosKato} and~\eqref{eq:cosLorentz},
$TT^*$ is a bounded operator from $\Kato^{\theta}_x L^{r'}_t$ (in particular 
$L^1_x L^2_t$) to its dual and from $L^{q', 2}_x L^{r'}_t$ to its dual.  Since 
we are interested in
mappings between dual spaces, the constant $C_{q'q}$ in~\eqref{eq:cosLorentz}
is uniformly bounded.  That establishes~\eqref{eq:StrichKato}
and~\eqref{maximal} for the imaginary part of
$\frac{e^{it\sqrt{H}}}{\sqrt{H}}P_c$.

The proof for the real part is nearly identical.  Let
$Tf = \frac{\cos(t\sqrt{H})}{\sqrt{H}}P_c f$.  Then the corresponding
$TT^*$ operator is defined by
$$
 \int_\R \frac {\cos(t\sqrt H) P_c}{\sqrt H} \frac {\cos(s\sqrt H) P_c}{\sqrt H} 
\dd s = \frac 
1 2 \int_{\R} \frac {\cos((t-s)\sqrt H) P_c}{H} + \frac {\cos((t+s)\sqrt H) 
P_c}{H} \dd s,
$$
which is once again controlled by~\eqref{eq:cosKato} and~\eqref{eq:cosLorentz} 
as above.

The reversed-norm Strichartz estimate for the wave equation with initial data
$f_0 \in \dot{H}^1$ follows readily from the decomposition
\begin{equation*}
\norm[\cos(t\sqrt{H})f_0][X] 
\leq \Bignorm[\frac{\cos(t\sqrt{H})P_c}{\sqrt{H}}][\B(L^2,X) ] 
\norm[(H P_c)^{1/2}][\B(\dot{H}^1,L^2)] \norm[f_0][\dot{H}^1]
\end{equation*}
with $X$ standing in for any of the target spaces 
$(\Kato^\theta)^*_xL^r_t \cap L^{q,2}_xL^r_t$.
Since $V \in \K_0$ is form-bounded with respect to the Laplacian, 
$H$ is a bounded map from $\dot{H}^1$ into $\dot{H}^{-1}$.  The spectral
projection $P_c$ removes a finite number of eigenfunctions, each of which
belong to $\dot{H}^1 \cap \dot{H}^{-1}$, hence $H P_c$ defines a positive
bounded quadratic form on $\dot{H}^1$.  Its square root then belongs to
$\B(\dot{H}^1, L^2)$.

The same proof based on the $T T^*$ method generalizes to the case of initial 
data in $\dot H^s \times \dot H^{s-1}$, for $\frac 12<s<\frac 32$. For 
$0<s<\frac 32$, $s \ne 1/2$, the propagator $\frac{\cos(t\sqrt{-\Delta})}{(-
\Delta)^s} \mid_{t \geq 0}$ has a radially symmetric kernel of the form
$$\begin{aligned}
K_s(r=|x-y|, t) &= \partial_t \frac {C_s} t \int_0^{\pi} \frac {t^2\sin 
\phi}{(r^2-2rt\cos\phi+t^2)^{\frac 3 2 - s}} \dd \phi \\
&= \partial_t \frac {C_s} r \bigg(\frac 1{|r-t|^{1-2s}} - \frac 1 {(r+t)^{1-
2s}}\bigg) \\
&= \frac {C_s} r \bigg(\frac {\sgn(t-r)} {|r-t|^{2-2s}} - \frac 1 {{(r+t)}^{2-
2s}}\bigg)
\end{aligned}$$
and a similar expression is retrieved in the case $s=\frac 12$ by means of a 
computation that involves logarithms. For $\frac 12<s \leq 1$ we obtain that
$$
\|K_s(r, t)\|_{L^1_t} \les_s \frac 1 {r^{2-2s}} \text{ and } \|K_s(r, 
t)\|_{L^{1/(2-2s),\infty}_t} \les_s \frac 1 r.
$$
Furthermore, for $1 < s < \frac 32$ there is a similar set of bounds
$\|K_s(r, t)\|_{L^{1/(3-2s), \infty}_t} \les_s 1$ and
$\|K_s(r, t)\|_{L^\infty_t} \les_s \frac 1 {r^{3-2s}}$,
and more generally, for $0 \leq \theta \leq 3-2s$,
$\|t^\theta K_s(r, t)\|_{L^\infty_t} \les_s \frac 1 {r^{3-2s-\theta}}$ .

Thus for $\frac 12<s\leq 1$ $\|K_s(r=|x-y|, t)\|_{L^1_t}$ is a bounded map from 
$L^1$ to $(\Kato^{2-2s})^\ast$ and from $\Kato^{2-2s}$ to $L^\infty$, while
$$
\|K_s(r=|x-y|, t)\|_{L^{1/(2-2s),\infty}_t} \in \B(L^1, \Kato^*) \cap \B(\Kato, 
L^\infty).
$$
For $1< s <\frac 32$, $0 \leq \theta \leq 3-2s$,
$$\begin{aligned}
\|K_s(r=|x-y|, t)\|_{L^{1/\theta, \infty}_t} &\les \|t^{\theta} K_s(r=|x-y|, 
t)\|_{L^\infty_t} \\
&\in \B(L^1, (\Kato^{3-2s-\theta})^\ast) \cap \B(\Kato^{3-2s-\theta}, L^\infty).
\end{aligned}$$

Since by Corollary \ref{general} $\rho^{\theta_2 - \theta_1} S^* \in 
\U_{(\Kato^{\theta_1})^*, (\Kato^{\theta_2})^*}$, these bounds carry over to the 
kernel of $\frac{\cos(t\sqrt{H})P_c}{H^s} \mid_{t \geq 0}$. Thus for $\frac 
12<s< 1$ ($s=1$ is the energy-critical case treated above), $0 \leq \theta \leq 
2-2s$, $\frac 1 p - \frac 1 q = 2s-1$, $1<p\leq q<\infty$,
\begin{align}
\lb{unu}\Big\|\int_{t'<t} \frac {\cos((t-t')\sqrt H)P_c}{H^s} F(t') \dd 
t'\Big\|_{(\Kato^{2-2s-\theta})^\ast_x L^{p, \sigma}_t} &\les 
\|F\|_{\Kato^\theta_x L^{p, \sigma}_t}, \\
\lb{doi}\Big\|\int_{t'<t} \frac {\cos((t-t')\sqrt H)P_c}{H^s} F(t') \dd 
t'\Big\|_{(\Kato^{1-\theta})^\ast_x L^{q, \sigma}_t} &\les \|F\|_{\Kato^\theta_x 
L^{p, \sigma}_t},
\end{align}
and for $1 < s < \frac 32$, $0 \leq \theta_1 + \theta_2 \leq 3-2s$, $\frac 1 p - 
\frac 1 q = \theta_1+\theta_2+2s-2$, $1< p \leq q < \infty$,
\begin{align}\lb{decay}
\Big\|\int_{t'<t} (t-t')^{3-2s-\theta_1-\theta_2} \frac {\cos((t-t')\sqrt 
H)P_c}{H^s} F(t') \dd t'\Big\|_{(\Kato^{\theta_2})^\ast_x L^\infty_t} \les 
\|F\|_{\Kato^{\theta_1}_x L^1_t}, \\
\lb{decay2}\Big\|\int_{t'<t} \frac {\cos((t-t')\sqrt H)P_c}{H^s} F(t') \dd 
t'\Big\|_{(\Kato^{\theta_2})^\ast_x L^{q, \sigma}_t} \les 
\|F\|_{\Kato^{\theta_1}_x L^{p, \sigma}_t}.
\end{align}
Interpolating between (\ref{unu}) and (\ref{doi}) we obtain that, for $\frac 12 
<s<1$, $2-2s\leq \theta_1 + \theta_2 \leq 1$ and $\frac 1 p - \frac 1 q = 
\theta_1 + \theta_2 + 2s - 2$, $1<p\leq q<\infty$,
$$
\bigg\|\int_{t'<t} \frac {\cos((t-t')\sqrt H)P_c}{H^s} F(t') \dd 
t'\bigg\|_{(\Kato^{\theta_2})^\ast_x L^{q, \sigma}_t} \les 
\|F\|_{\Kato^{\theta_1}_x L^{p, \sigma}_t}.
$$
A Lorentz space version of the above inequality is, for $\frac 12<s<1$, $0 \leq 
\frac 1 r - \frac 1 {\tilde r} \leq 2s-1$, $\frac 1 r - \frac 1 {\tilde r} = 1 - 
\frac 3 p + \frac 3 q + 2s$, $1 < \tilde r \leq r < \infty$, $1 < p \leq q < 
\infty$,
$$
\bigg\|\int_{t'<t} \frac {\cos((t-t')\sqrt H)P_c}{H^s} F(t') \dd 
t'\bigg\|_{L^{q, \sigma}_x L^{r, \tilde \sigma}_t} \les \|F\|_{L^{p, \sigma}_x 
L^{\tilde r, \tilde \sigma}_t}.
$$
Likewise, a Lorentz space version of (\ref{decay2}) is that for $1<s<\frac 32$, 
$2s-2 \leq \frac 1 r - \frac 1 {\tilde r} \leq 1$, $\frac 1 r - \frac 1 {\tilde 
r} = 1 - \frac 3 p + \frac 3 q + 2s$, $1 < \tilde r \leq r < \infty$, $1<p \leq 
q<\infty$
$$
\bigg\|\int_{t'<t} \frac {\cos((t-t')\sqrt H)P_c}{H^s} F(t') \dd 
t'\bigg\|_{L^{q, \sigma}_x L^{r, \tilde \sigma}_t} \les \|F\|_{L^{p, \sigma}_x 
L^{\tilde r, \tilde \sigma}_t}
$$
and for $1<s<\frac 32$, $\frac 3 p - \frac 3 q = 2s$, $1<p \leq q<\infty$,
$$
\bigg\|\int_{t'<t} \frac {\cos((t-t')\sqrt H)P_c}{H^s} F(t') \dd 
t'\bigg\|_{L^{q, \sigma}_x L^\infty_t} \les \|F\|_{L^{p, \sigma}_x L^1_t}.
$$

Identifying dual spaces in the previous inequalities, by the $T T^\ast$ method 
we arrive at the following estimates for $\frac{\cos(t\sqrt H) P_c}{H^{s/2}}$ 
and $\frac{\sin(t\sqrt H) P_c}{H^{s/2}}$: when $\frac 1 2<s< 1$, $1-s \leq 
\theta \leq \frac 1 2$, $\theta+\frac 1 r = \frac 3 2-s$ and when $6 \leq 
q\leq\frac 3 {1-s}$, $\frac 3 q + \frac 1 {\tilde r} = \frac 3 2 - s$
$$
\Big\|\frac{e^{it\sqrt H} P_c}{H^{s/2}} f\Big\|_{(\Kato^{\theta})^\ast_x L^{r, 
2}_t \cap L^{q, 2}_x L^{\tilde r, 2}_t} \les \|f\|_{L^2}.
$$
For $1 < s < 3/2$, one obtains that when $0 \leq \theta \leq \frac 3 2 - s$ 
(note this includes $L^\infty$ when $\theta=0$), $\theta+\frac 1 r = \frac 3 2-
s$, $r< \infty$ and when $\frac 6 {3-2s} < q < \infty$, $\frac 3 q + \frac 1 
{\tilde r} = \frac 3 2 - s$, $\tilde r < \infty$,
$$
\Big\|\frac{e^{it\sqrt H} P_c}{H^{s/2}} f\Big\|_{(\Kato^{\frac 3 2 - s})^\ast_x 
L^\infty_t \cap L^{\frac 6 {3-2s}, 2}_x L^\infty_t \cap (\Kato^\theta)^\ast_x 
L^{r, 2}_t \cap L^{q, 2}_x L^{\tilde r, 2}_t} \les \|f\|_{L^2}.
$$

The regular Strichartz inequalities~\eqref{eq:Strichartz1} can be proved
in the manner of Ginibre-Velo~\cite{ginebre}. In the energy-critical case $s=1$, 
thanks to the dispersive bound~\eqref{eq:dispLorentz} and to Young's inequality, 
we obtain that when $\frac 1 r - \frac 1 {\tilde r} = \theta_1 + \theta_2$, 
$1<r< \tilde r<\infty$,
$$
\bigg\|\int_{t'<t} \frac {\cos((t-t')\sqrt H)P_c}{H} F(t') \dd 
t'\bigg\|_{L^{\tilde r, \tilde \sigma}_t (\Kato^{\theta_2})^*_x} \les 
\|F\|_{L^{r, \tilde \sigma}_t \Kato^{\theta_1}_x}
$$
and $\frac 1 r - \frac 1 {\tilde r} = 3 - \frac 3 p + \frac 3 q$, $1 <p \leq q < 
\infty$, $1<r < \tilde r < \infty$
$$
\bigg\|\int_{t'<t} \frac {\cos((t-t')\sqrt H)P_c}{H} F(t') \dd 
t'\bigg\|_{L^{\tilde r, \tilde \sigma}_t L^{q, \sigma}_x} \les \|F\|_{L^{r, 
\tilde \sigma}_t L^{p, \sigma}_x}.
$$
From (\ref{decay}) we infer that for $1 < s < \frac 32$, $\frac 3 p - \frac 3 q 
= \frac 1 r + 2s$, and $1<p \leq q<\infty$,
$$
\Big\|t^{1/r} \frac{\cos(t\sqrt H)P_c}{H^s} f \Big\|_{L^{q, \sigma}_x 
L^\infty_t} \les \|f\|_{L^{p, \sigma}_x}.
$$
Thus for $1<s<\frac 32$, $\frac 1 r - \frac 1 {\tilde r} = 1 - \frac 3 p + \frac 
3 q + 2s$, $1 <p \leq q < \infty$, $1<r < \tilde r < \infty$,
$$
\bigg\|\int_{t'<t} \frac {\cos((t-t')\sqrt H)P_c}{H^s} F(t') \dd 
t'\bigg\|_{L^{\tilde r, \tilde \sigma}_t L^{q, \sigma}_x} \les \|F\|_{L^{r, 
\tilde \sigma}_t L^{p, \sigma}_x}
$$
and for $1<s<\frac 32$, $\frac 1 r - \frac 1 {\tilde r} = \theta_1+\theta_2+2s-
2$, $1<r < \tilde r < \infty$,
$$
\bigg\|\int_{t'<t} \frac {\cos((t-t')\sqrt H)P_c}{H^s} F(t') \dd 
t'\bigg\|_{L^{\tilde r, \tilde \sigma}_t (\Kato^{\theta_2})^*_x} \les 
\|F\|_{L^{r, \tilde \sigma}_t \Kato^{\theta_1}_x}.
$$

Identifying dual spaces, by a $T T^*$ argument we obtain that in the energy-
critical case, for $0<\theta< \frac 1 2$, $\theta+ \frac 1 r = \frac 1 2$ and 
for $6< q < \infty$, $\frac 3 q + \frac 1 {\tilde r} = \frac 1 2$,
$$
\Big\|\frac{e^{it\sqrt H} P_c}{H^{s/2}} f\Big\|_{L^{r, 2}_t (\Kato^\theta)^*_x 
\cap L^{\tilde r, 2}_t L^{q, 2}_x} \les \|f\|_2.
$$
In addition, the $L^2$ conservation law also makes this expression bounded in 
$L^\infty_t L^{6, 2}_x$.

When $1<s<3/2$, the same $T T^*$ argument shows that for $0 \leq \theta < \frac 
3 2 - s$, $\theta+\frac 1 r = \frac 3 2 - s$ and $\frac 6 {3-2s} < q < \infty$, 
$\frac 3 q + \frac 1 {\tilde r} = \frac 3 2 - s$
$$
\Big\|\frac{e^{it\sqrt H} P_c}{H^{s/2}} f\Big\|_{L^{r, 2}_t (\Kato^\theta)^*_x 
\cap L^{\tilde r, 2}_t L^{q, 2}_x} \les \|f\|_2.
$$
By $L^2$ conservation, this expression is also bounded in $L^\infty_t L^{\frac 6 
{3-2s}, 2}_x$. Setting $\theta=0$ we also obtain a $L^{\frac 2 {3-2s}, 2}_t 
L^\infty_x$ bound.

For the remaining case $0 \leq s \leq 1$, separate the real and imaginary
parts of $H^{-1/2}e^{it\sqrt{H}}P_c$ as above.  When $s=1$,
a $TT^*$ construction leads to the operator defined by convolution in time with
$\frac{\cos(t\sqrt{H})}{H}P_c$.  Thanks to the dispersive
bound~\eqref{eq:dispLorentz} and Young's inequality, it is a bounded linear
map from $L^{r'}_t L^{q',2}_x$  to its dual so long as $(r,q)$ is 
Strichartz-admissible with $r > 2$.  

When $s=0$ the functional calculus of $H$
provides a conservation law in $L^2(\R^3)$.  Moreover, every imaginary
power of $(H P_c)$ is a partial isometry on $L^2$. Thus~\eqref{eq:Strichartz1}
holds whenever $\Re[s] = 1$ and $\Re[s] = 0$, uniformly with respect to 
the imaginary part of $s$.  Complex interpolation fills in the intermediate 
cases $0<s<1$.

Estimates~\eqref{eq:Strichartz} for the solution of an initial-value problem
are equivalent to~\eqref{eq:Strichartz1}, in view of Lemma~\ref{lem:norm_equiv}.
\end{proof}

\begin{proof}[Proof of Theorem \ref{thm:sharpStrichartz}]
It suffices to prove the second inequality~\eqref{eq:sharpStrichartz}.
Then~\eqref{eq:sharpStrichartz1} follows by the equivalence of homogeneous
and perturbed Sobolev spaces set forth in Lemma~\ref{lem:norm_equiv}.

The main estimates are another consequence of the formulas
\begin{align*}
\cos(t\sqrt{H})P_c f_0 &=
\frac{1}{2\pi i} \int_{-\infty}^\infty (e^{it\lambda}+e^{-it\lambda})
(I +\hat{T}^-(\lambda)^*)^{-1}\lambda R_0^+(\lambda^2) f_0\,d\lambda \\
\frac{\sin(t\sqrt{H})P_c}{\sqrt{H}}f_1 &= -\frac{1}{2\pi}
\int_{-\infty}^\infty (e^{it\lambda} - e^{-it\lambda})
 (I + \hat{T}^-(\lambda)^*)^{-1}R_0^+(\lambda^2)f_1\,d\lambda
\end{align*}
that appeared in the proof of Theorem~\ref{main_theorem}.  Returning to the
notation there, 
\begin{align} 
\cos(t\sqrt{H})P_c f_0(t,x) &= 
-i\big(S^*(\lambda R_0^+)^\vee f_0)(t,x) 
+ (S^*(\lambda R_0^+)^\vee f_0)(-t,x)\big)  \label{eq:S*cosine}\\
\label{eq:S*sine}
\frac{\sin(t\sqrt{H})P_c}{\sqrt{H}} f_1(t,x)
&= - \big((S^* \check{R}_0^+ f_1)(t,x) - 
 (S^*\check{R}_0^+ f_1)(-t,x)\big).
\end{align}
Here $\check{R}_0^+(t,x,y)$ is the sine propagator of the free wave
equation restricted to $t \leq 0$ and $(\lambda R_0^+)^\vee(t,x,y)$
is the free cosine propagator.  

It is well known from~\cite{tao} that the
free evolution satisfies the Strichartz estimates
$\norm[\check{R}_0^+f_1][L^4_{x,t}] \les \norm[f_1][\dot{H}^{-1/2}]$ and
$\norm[(\lambda R_0^+)^\vee f_0][L^4_{x,t}] \les \norm[f_0][\dot{H}^{1/2}]$.
Because $V$ is assumed to belong to $L^{3/2,1}(\R^3)$,
Theorem~\ref{thm:goal_lattice} indicates that 
$S = (I - T^-)^{-1} \in \U_{L^{4/3}}$.  Then $S^* \in \U_{L^4}$ is a bounded 
operator
on $L^4_{x,t}$.  This establishes the bound 
\begin{equation*}
\Bignorm[\cos(t\sqrt{H})P_c f_0 + \frac{\sin(t\sqrt{H})P_c}{\sqrt{H}} f_1][
L^4_tL^4_x]
\les \norm[f_0][\dot{H}^{1/2}] + \norm[f_1][\dot{H}^{-1/2}]
\end{equation*}
The other endpoint of the range,
\begin{equation*}
\bignorm[\cos(t\sqrt{H})P_c f_0 + \frac{\sin(t\sqrt{H})P_c}{\sqrt{H}} f_1][
L^\infty_t L^2_x] \les \norm[f_0][\dot H^1] + \norm[f_1][2],
\end{equation*}
follows directly from the spectral theorem applied to $H$ and from
Lemma~\ref{lem:norm_equiv}.  Complex interpolation 
proves~\eqref{eq:sharpStrichartz} in the intermediate cases $0 < s < \frac12$.

\end{proof}

\begin{proof}[Proof of Theorem \ref{limit}]
It suffices to show that $f(t,x) = \cos(t\sqrt{H})P_c f_0 + 
\frac{\sin(t\sqrt{H})}{\sqrt{H}}P_c f_1$ can be approximated in
$L^{6}_x L^\infty_t$ by functions that are uniformly continuous with respect
to $t$.  Under the assumptions of Theorem~\ref{main_theorem} there are
a finite number of eigenvalues, and each eigenfunction $\psi_j$ belongs to 
$\dot{H}^1(\R^3) \cap \dot{H}^{-1}(\R^3)$.  Then $P_c$
acts boundedly on both $f_0 \in \dot{H}^1$ and $f_1 \in L^2$, and the
coefficients of $\psi_j$ in the full (i.e. unprojected)
evolution are a linear combination of $\cosh(t E_j)$ and $\sinh(t E_j)$,
which does not influence the short-time continuity of solutions.

By 
Lemma~\ref{lem:norm_equiv} the cross-sections $f(t,\,\cdot\,)$ are well defined
elements of $\dot{H}^1$ and converge in norm to $f(0,\,\cdot\,) = P_c f_0$ as
$t$ approaches zero.

We look at functions of the form
$\tilde{f} = \cos(t\sqrt{\tilde{H}})\tilde{P}_c\tilde{f}_0
+ \frac{\sin(t\sqrt{H})}{\sqrt{H}}P_c\tilde{f}_1$, where $\tilde{f}_0,
\tilde{f}_1 \in C^\infty_c(\R^3)$ are smooth approximations of the initial data
and $\tilde{H} = -\Delta + \tilde{V}$ has a smooth compactly supported
approximation of the potential.  The projection $\tilde{P_c}$ is taken according
to the spectral measure of $\tilde{H}$.

Both the continuity of $\tilde{f}$ and its approximation properties are derived
from formulas~\eqref{eq:S*cosine} and~\eqref{eq:S*sine}.  
Recall once more that $\check{R}_0^+(t,x,y)$ is the backwards sine propagator
of the free wave equation.  Choosing 
$\tilde{f}_1 \in C^\infty_c(\R^3)$
insures that $\check{R}_0^+ \tilde{f}_1$ is Lipschitz continuous where it
crosses the plane $\{t=0\}$ and  is globally bounded with bounded
derivative.  The action of $S^* \in \U_{L^\infty}$ preserves
$L^\infty_{x,t}$ norms, but thanks to its convolution structure in the $t$
variable it also preserves Lipschitz continuity in the $t$ direction.
Precisely,
\begin{align*}
|S^*g(t+\delta,x) - S^*g(t,x)| &= 
|S^*(g(\,\cdot\,+\delta,x) - g(\,\cdot\,,x))| \\
&\leq \norm[S^*][\U_{L^\infty}]
\norm[g(\,\cdot\,+\delta,x) - g(\,\cdot\,,x)][L^\infty_{x,t}].
\end{align*}
It follows that $\frac{\sin(t\sqrt{H})}{\sqrt{H}}P_c \tilde{f}_1$ is
uniformly Lipschitz.  The density of smooth test functions in $L^2$
combined with~\eqref{eq:StrichSobolev} shows that it can be used to
approximate $\frac{\sin(t\sqrt{H})}{\sqrt{H}}P_c f_1$ in the space
$L^6_xL^\infty_t$.

Proving continuity of the cosine evolution is slightly more complicated.
Within the formula
\begin{equation*}
\cos(t\sqrt{\tilde{H}})\tilde{P}_c \tilde{f}_0 (t,x)= 
-i\big((\tilde{S}^*(\lambda R_0^+)^\vee \tilde{f}_0)(t,x) 
+ (\tilde{S}^*(\lambda R_0^+)^\vee \tilde{f}_0)(-t,x)\big)
\end{equation*}
one encounters the obstruction that $(\lambda R_0^+)^\vee \tilde{f}_0$
has a jump discontinuity across the plane $\{t=0\}$ of size $\tilde{f}_0(x)$,
even though it is smooth elsewhere.  Imitating the
argument above with $\tilde{S}^* \in \U_{L^\infty}$ only shows that
$\tilde{S}^*(\lambda R_0^+)^\vee \tilde{f}_0$ has bounded variation in the
$t$ direction.  

Recall that the Fourier transform of $\tilde{S}^*$ is 
$(I + R_0^+(\lambda^2)\tilde{V})^{-1}$.  Then
\begin{equation*} 
i \cos(t\sqrt{\tilde{H}})\tilde{P}_c \tilde{f}_0\Big|_{t<0} =
\tilde{S}^*(\lambda R_0^+)^\vee \tilde{f}_0 = 
(\lambda R_0^+)^\vee \tilde{f}_0 - \tilde{S}^* \check{R}_0^+ \tilde{V}
(\lambda R_0^+)^\vee \tilde{f}_0.
\end{equation*}
This is more or less a restatement of Duhamel's formula over the half-space
$\{t < 0\}$, since $\check{R}_0^+$ and $(\lambda R_0^+)^\vee$ are the
backward fundamental solutions of the free wave equation.
On the assumption that $\tilde{V} \in C^\infty_c(\R^3)$, we have that
$\tilde{V} (\lambda R_0^+)^\vee \tilde{f}_0$ is bounded with bounded derivatives
when $t < 0$, and consequently 
$\check{R}_0^+(\lambda R_0^+)^\vee \tilde{f}_0$ is uniformly
Lipschitz.

From here the previous argument applies to show that $\cos(t\sqrt{\tilde{H}})
\tilde{P}_c \tilde{f}_0$ is uniformly Lipschitz when $t < 0$.  It is an even
function of $t$, thus the same Lipschitz constant is valid over the entire range
$t \in (-\infty, \infty)$.

By~\eqref{eq:StrichSobolev}, $\tilde{f}_0$ can be chosen so that 
$\cos(t\sqrt{\tilde{H}})\tilde{P}_c\tilde{f}_0$ is a close approximation to
$\cos(t\sqrt{\tilde{H}})\tilde{P}_c f_0$.  The remaining task is to show that
the spectral multiplier $\cos(t(-\Delta + V))$ varies continuously with
$V \in \Kato_0$ in the absence of nonnegative eigenvalues and resonances.

This is easiest to accomplish when $V \in L^{3/2,1}$.  In that case both
$(T^-)^*$ and $S^*$ belong to $\U_{L^{6,2}}$ by Theorem~\ref{thm:goal_lattice}.
Furthermore $\tilde{S}^* = (\1 + (\tilde{T}^-)^*)^{-1}$ has a comparable norm
in $\U_{L^{6,2}}$ provided $\norm[V - \tilde{V}][L^{3/2,1}]$ is sufficiently 
small.
Then the difference between $S^*$ and $\tilde{S}^*$ is controlled by
\begin{equation*}
\norm[S^* - \tilde{S}^*][\U_{L^6}] = 
\norm[S^*(T^- - \tilde{T}^-)\tilde{S}^*][\U_{L^6}]
\les \norm[S^*][\U_{L^6}]^2\norm[V-\tilde{V}][L^{3/2,1}].
\end{equation*}
Since $\U_{L^6} \subset \B(L^6_xL^\infty_t)$ it follows that 
\begin{equation*}
i \cos(t\sqrt{\tilde{H}})\tilde{P}_c f_0   -
i \cos(t\sqrt{H})P_c f_0 \Big|_{t<0} =
(S^* - \tilde{S}^*)(\lambda R_0^+)^\vee f_0
\end{equation*}
has $L^6_xL^\infty_t$ norm controlled by $\norm[S^*][\U_{L^6}]^2
\norm[V-\tilde{V}][L^{3/2,1}] \norm[f_0][\dot{H}^1]$.

In the more general case where $V \in \Kato_0$, bounds in $\U_{L^6}$ are not
directly available.  The difference between the the evolution of $H$ and 
$\tilde{H}$
is instead estimated by $TT^*$ arguments and interpolation, similar to the proof 
of 
Theorem~\ref{main_theorem}.  Set
\begin{equation*}
Tf_0 = (S^*- \tilde{S}^*)\big(\cos(t\sqrt{-\Delta})f_0\big|_{t<0}\big),
\end{equation*}
with $f_0 \in \dot{H}^1(\R^3)$, which makes
\begin{equation*}
TT^* = \frac12 (S^* - \tilde{S}^*)\Big(
\frac{\cos((t+s)\sqrt{-\Delta})-\cos((t-s)\sqrt{-\Delta})}{-\Delta}\Big|_{s,t < 
0}
\Big)(S - \tilde{S}).
\end{equation*}
The central operator maps $L^1_{x,t}$ to $\Kato^*_x L^\infty_t$. 
Meanwhile $\norm[S - \tilde{S}][\U_{L^1}]$ is controlled by
$\norm[S][\U_{L^1}]^2 \norm[V- \tilde{V}][\Kato]$ as above, and similarly
$\norm[S^* - \tilde{S}^*][\U_{\Kato^*}] \les
\norm[S^*][\U_{\Kato^*}]^2 \norm[V - \tilde{V}][\Kato]$. Thus
the norm of $TT^*: L^1_{x,t} \to \Kato^*_xL^\infty_t$ is less than a
constant times $\norm[V-\tilde{V}][\Kato]^2$.

Taking adjoints, the norm of $TT^*: \Kato_x L^1_t \to L^\infty_{x,t}$
has the same bound.  Real interpolation provides the desired mapping
estimate between $L^{6/5}_xL^1_t$ and $L^6_xL^\infty_t$, with the
conclusion that
\begin{equation*}
\norm[\cos(t\sqrt{H})P_cf_0 - \cos(t\sqrt{\tilde{H}})\tilde{P}_cf_0][
L^6_xL^\infty_t] \les \norm[S][\U_{L^1}]\norm[S][\U_\Kato]
\norm[V-\tilde{V}][\Kato] \|f_0\|_{\dot{H}^1}.
\end{equation*}

\end{proof}

\end{document}